\documentclass[12pt]{article}

\usepackage{a4wide}
\usepackage{graphicx}
\usepackage{amsmath}
\usepackage{amssymb}
\usepackage{amsthm}
\usepackage{enumerate}
\usepackage{hyperref}
\usepackage{tikz}

\usetikzlibrary{shapes.misc}
\tikzset{xx/.style={cross out, draw=red, very thick, inner sep=0, minimum size=1mm},
xxx/.style={cross out, draw=red, inner sep=0, minimum size=1.2mm},
o/.style={circle, draw=gray, inner sep=0, minimum size=1mm},
oo/.style={circle, fill, inner sep=0, minimum size=2mm},
ooo/.style={circle, fill=green!50!black, inner sep=0, minimum size=2mm},
not/.style={cross out, draw=red, very thick, inner sep=0, minimum size=1mm},
vertex/.style={circle, fill, inner sep=0, minimum size=2.5mm},
edge/.style={very thick},
fault line/.style={red, line width=1.5mm, line cap=rect},
fault line thin/.style={red, line width=1mm, line cap=rect}}

\newtheorem{theorem}{Theorem}[section]
\newtheorem{definition}[theorem]{Definition}
\newtheorem{proposition}[theorem]{Proposition}
\newtheorem{lemma}[theorem]{Lemma}
\newtheorem{claim}[theorem]{Claim}

\newtheorem{conjecture}[theorem]{Conjecture}

\newtheorem{problem}[theorem]{Problem}
\newtheorem{observation}[theorem]{Observation}
\usepackage[margin=20pt,small,bf]{caption}

\newcommand{\Z}{\mathbb{Z}}

\newcommand{\lset}[1]{\left\{#1\right\}}
\newcommand{\set}[2]{\left\{#1 : #2 \right\}}
\newcommand{\abs}[1]{\left|#1\right|}
\newcommand{\seq}[1]{({#1}_i)_{i \in \Z}}
\newcommand{\rev}[1]{({#1}_{-i})_{i \in \Z}}

\begin{document}

\title{Graphs without proper subgraphs of minimum  degree~3 and short cycles}

\author{\large{Lothar Narins}\thanks{Research supported by the Research Training Group Methods for Discrete Structures and the
Berlin Mathematical School.}, \large{Alexey Pokrovskiy}\thanks{Research supported by the Research Training Group Methods for Discrete Structures.}, \large{Tibor Szab\'o} \thanks{Research partially supported by DFG within the Research Training Group Methods for
Discrete Structures.}\\ \\
\normalsize Department of Mathematics,
\\ \normalsize  {Freie Universit\"at,} 
\\ \normalsize  Berlin, Germany. 
 }

\maketitle

\begin{abstract}
We study graphs on $n$ vertices which have $2n-2$ edges and no proper induced subgraphs of minimum degree $3$.  Erd\H{o}s, Faudree, Gy\'arf\'as, and Schelp conjectured that such graphs always have cycles of lengths $3,4,5,\dots, C(n)$ for some function $C(n)$ tending to infinity.  We disprove this conjecture, resolve a related problem about leaf-to-leaf
path lengths in trees, and characterize graphs with $n$ vertices and $2n-2$ edges, containing no proper
subgraph of minimum degree $3$. 
\end{abstract}

\section{Introduction}
A simple exercise in graph theory is to show that every graph $G$ with $n$ vertices and at least $2n-2$ edges must have an induced subgraph with minimum degree $3$. Moreover, this statement is best possible: there are several constructions with $2n-3$ edges which do not have this property.
So every graph with $n$ vertices and $2n-2$ edges must contain an induced subgraph with minimum degree $3$, however this subgraph might be the whole graph.  A subgraph $H$ of $G$ is called {\em proper} if $H\neq G$.  See Figure~\ref{ExampleGraphs} for two examples of graphs with $2|G|-2$ edges but no proper induced subgraphs of minimum degree $3$.  The first of these, has an even stronger property---it has no proper induced or non-induced subgraphs with minimum degree $3$.  On the other hand, the second example has a proper non-induced subgraph with minimum degree $3$ formed by removing the edge between the two vertices of degree $4$.

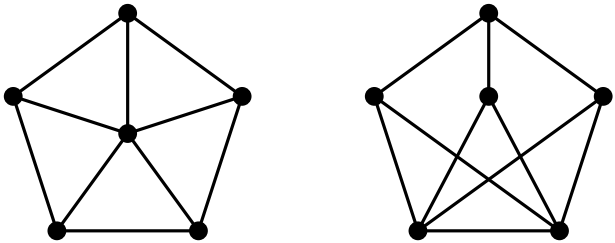
\begin{figure}[hbt]
  \centering
  	\begin{tikzpicture}[scale=0.8]
  		\draw[edge] (0,2) -- (2*0.951,2*0.309) -- (2*0.588,2*-0.809) -- (2*-0.588,2*-0.809) -- (2*-0.951,2*0.309) -- cycle;
  		\draw[edge] (0,0) -- (0,2);
  		\draw[edge] (0,0) -- (2*0.951,2*0.309);
  		\draw[edge] (0,0) -- (2*0.588,2*-0.809);
  		\draw[edge] (0,0) -- (2*-0.588,2*-0.809);
  		\draw[edge] (0,0) -- (2*-0.951,2*0.309);
  	
  		\node[vertex] at (0,0) {};
  		\node[vertex] at (0,2) {};
  		\node[vertex] at (2*0.951,2*0.309) {};
  		\node[vertex] at (2*0.588,2*-0.809) {};
  		\node[vertex] at (2*-0.588,2*-0.809) {};
  		\node[vertex] at (2*-0.951,2*0.309) {};
  		
  		\draw[edge] (6+0,2) -- (6+2*0.951,2*0.309) -- (6+2*0.588,2*-0.809) -- (6+2*-0.588,2*-0.809) -- (6+2*-0.951,2*0.309) -- cycle;
  		\draw[edge] (6+0,2*0.309) -- (6+0,2);
  		\draw[edge] (6+0,2*0.309) -- (6+2*0.588,2*-0.809);
  		\draw[edge] (6+0,2*0.309) -- (6+2*-0.588,2*-0.809);
  		\draw[edge] (6+2*0.951,2*0.309) -- (6+2*-0.588,2*-0.809);
  		\draw[edge] (6+2*0.588,2*-0.809) -- (6+2*-0.951,2*0.309);
  		
  		\node[vertex] at (6+0,2*0.309) {};
  		\node[vertex] at (6+0,2) {};
  		\node[vertex] at (6+2*0.951,2*0.309) {};
  		\node[vertex] at (6+2*0.588,2*-0.809) {};
  		\node[vertex] at (6+2*-0.588,2*-0.809) {};
  		\node[vertex] at (6+2*-0.951,2*0.309) {};
  	\end{tikzpicture}
  \caption{Two examples of graphs on $6$ vertices with $10$ edges and no proper induced subgraphs with minimum degree $3$.} \label{ExampleGraphs}
\end{figure}

In this paper we will study graphs with $n$ vertices $2n-2$ edges which have no proper induced subgraphs with minimum degree $3$. Following Bollob\'as and Brightwell~\cite{BB} we call such graphs  {\em degree $3$-critical}. It is easy to see that graphs with $n$ vertices and
at least $2n-1$ edges contain a {\em proper} degree $3$-critical subgraph.
Erd\H os (cf \cite{EFRS})  conjectured that 
they should contain a degree $3$-critical subgraph not only on at most $n-1$, 
but on at most $(1-\epsilon)n$ vertices, for some constant $\epsilon >0$.
Degree $3$-critical graphs are closely related to several other interesting classes of graphs. For example, they have the property that all their proper subgraphs are 2-degenerate 
(where a graph is defined to be {\em $2$-degenerate} if it has no subgraph of 
minimum degree $3$).  Also notice that degree $3$-critical graphs certainly have no proper subgraphs $H$ with $2|H|-2$ edges. Graphs with $2n-2$ edges and no proper subgraphs $H$ with $2|H|-2$ edges have a number of interesting properties. They are {\em rigidity circuits}: by a theorem of Laman, removing any edge from such a graph produces a graph $H$ which is \emph{minimally rigid in the plane}, i.e., any embedding of it into the plane where the vertices are substituted by joints and the edges by rods produces a rigid structure, but no proper subgraph of $H$ has 
this property.  Furthermore, by
a special case of a theorem of Nash-Williams 
these graphs are exactly the ones that are the union of two disjoint spanning trees and
Lehman's Theorem characterizes them as the minimal graphs to win the 
so-called {\em connectivity game} on. That is,
with two players alternately occupying the edges of $G$, the player playing second is able to occupy a spanning tree.

The study of degree $3$-critical graphs was initiated by Erd\H{o}s, Faudree, Gy\'arf\'as, and   Schelp~\cite{EFGS}, where they investigated the possible cycle lengths. They showed that degree $3$-critical graphs on $n\geq 5$ vertices 
always contain a cycle of length $3$, $4$, and 
$5$, as well as a cycle of length at least $\lfloor \log_2 n\rfloor$, but not 
necessarily of length more than $\sqrt{n}$.  Bollob\'as and Brightwell \cite{BB} resolved asymptotically the question of how short the longest cycle length in degree $3$-critical graphs can  be. 
They showed that every degree $3$-critical graph contains a cycle of length at least 
$4\log_2 n-o(\log n)$ and constructed degree $3$-critical graphs 
with no cycles of length more than $4\log_2 n+ O(1)$.
Erd\H{o}s, et al.~\cite{EFGS} made the following conjecture about possible cycle lengths in degree $3$-graphs.
\begin{conjecture}[Erd\H{o}s, Faudree, Gy\'arf\'as, and   Schelp, \cite{EFGS}]\label{ErdosConjecture}
There is an increasing function $C(n)$ such that the following holds such that
every degree $3$-critical graph on $n$ vertices  
contains all cycles of lengths $3, 4, 5, 6, \dots, C(n)$.
\end{conjecture}

A historical remark must be made here.  The exact phrasing of Conjecture~\ref{ErdosConjecture} in \cite{EFGS} is not quite what is stated above.  
In \cite{EFGS} first a class of graphs, $G^*(n, m)$, is defined as ``the set of graphs with $n$ vertices, $m$ edges, and with the property that no proper subgraph has minimum degree $3$.''  Then Conjecture~\ref{ErdosConjecture} is stated as ``If $G\in G^*(n,2n-2)$, then $G$ contains all cycles of length at most $k$ where $k$ tends to infinity.''  Notice that the word ``induced'' is not present in the original formulation.  However a careful reading of \cite{EFGS} shows that in that paper ``proper subgraph'' implicitly must mean ``proper induced subgraph''.  Indeed many of the constructions given in \cite{EFGS} (such as Examples 1, 2, 3, 5, and 6 on pages 197-201) of graphs which have ``no proper subgraphs of minimum degree $3$'' actually do have proper non-induced subgraphs with minimum degree $3$.  In addition, one can check that all the results and proofs given in \cite{EFGS} concerning graphs with ``no proper subgraphs of minimum degree $3$'' hold also for graphs with ``no proper induced subgraphs of minimum degree $3$''.  Therefore, it is plausible to assume that the word ``induced'' should be present in the statement of Conjecture~\ref{ErdosConjecture}. This also coincides with the interpretation of the concept 
in the paper of
Bollob\'as and Brightwell~\cite{BB}. 

Consequently throughout most of this paper will study Conjecture~\ref{ErdosConjecture} as it is stated above.  However, for the sake of completeness, in Section~\ref{NonInducedSection} we will diverge and consider the special case of Conjecture~\ref{ErdosConjecture} when $G$ contains neither induced nor non-induced subgraphs with minimum degree $3$.

The  main result of this paper is a disproof of Conjecture~\ref{ErdosConjecture}.  We prove the following.
\begin{theorem}\label{Counterexamples}
There is an infinite sequence of degree $3$-critical graphs $(G_n)_{n=1}^{\infty}$ which
do not contain a cycle of length $23$.
\end{theorem}

In the process of proving this theorem, we will naturally arrive to a question of independent interest, concerning the various {\em leaf-leaf path} lengths 
(i.e., the lengths of paths going between two leaves) that must occur in a tree.
Obviously, if $T$ is just a path, then $T$ only has a single leaf-leaf path.  
However if $T$ has no degree $2$ vertices, then one would expect $T$ to have many different leaf-leaf path lengths.  Of particular relevance to Conjecture~\ref{ErdosConjecture} will be even $1$-$3$ trees.  A tree is called \emph{even} if all of its leaves are in the 
same class of the tree's unique bipartition and a tree is called a {\em $1$-$3$-tree} 
if every vertex has degree $1$ or $3$. On our way towards the proof of 
Theorem~\ref{Counterexamples} we determine the smallest even number 
which does not occur as a leaf-leaf path in 
every even $1$-$3$-tree.  
\begin{theorem}\label{TreeLengths}
\begin{enumerate}[(i)]
\item There is an integer $N_0$ such that every even $1$-$3$ tree $T$ with 
$|T|\geq N_0$ contains leaf-leaf paths of lengths $0,2,4,\dots, 18$.
\item There is an infinite family of even $1$-$3$ trees $(T_n)_{n=1} ^{\infty}$, such that $T_n$ contains no leaf-leaf path of length $20$.
\end{enumerate} 
\end{theorem}



Part $(ii)$ of Theorem~\ref{TreeLengths} will be used to construct our counterexample
to Conjecture~\ref{ErdosConjecture}, while part $(i)$ shows
that our method, as is, can not deliver a stronger counterexample. 
Hence it would be interesting to determine the shortest cycle length which is not present in every sufficiently large degree $3$-critical graph. Theorem~\ref{Counterexamples} shows that this number is at most $23$, while Erd\H{o}s et al.~\cite{EFGS} showed that it is at least $6$. They  also mention that their methods could be extended to work for $7$. 
In  Section~\ref{C6Section} we verify their statement, by giving a short proof that every 
degree $3$-critical graph must contain $C_6$.

Finally, we revisit Conjecture~\ref{ErdosConjecture} 
with the word ``induced'' removed from the definition of degree $3$-critical.  
We characterize all $n$-vertex graph with $2n-2$ edges and no proper (not necessarily induced) subgraph with minimum degree $3$ and 
show that the conjecture is true for them in a much stronger form.
\begin{theorem}\label{NonInduced}
Let $G$ be a graph with $n$ vertices, $2n-2$ edges and no proper 
subgraph with minimum degree $3$.  Then $G$ is pancyclic, that is, it 
contains cycles of length $i$ for every $i=3,4,5,\dots,$ and $n$.
\end{theorem}
Theorem~\ref{NonInduced} will follow from a structure theorem which we shall prove about graphs with  $n$ vertices, $2n-2$ edges and no proper (not necessarily induced) subgraphs with minimum degree $3$.  It will turn out that there are only two particular 
families of graphs satisfying these conditions.  
One of them is the family of wheels and the other is a family of graphs obtained from a wheel by replacing one of its edges with a certain other graph.  

The structure of this paper is as follows.   In Section~\ref{CounterexampleSection} we construct our counterexamples to Conjecture~\ref{ErdosConjecture} via proving part $(ii)$ of Theorem~\ref{TreeLengths} and Theorem~\ref{Counterexamples}.  In Section~\ref{TreeLengthsSection} we study necessary leaf-leaf path lengths in even $1$-$3$ 
trees and prove part $(i)$ Theorem~\ref{TreeLengths}. 
In Section~\ref{NonInducedSection} we prove the weakening of Conjecture~\ref{ErdosConjecture} when the word ``induced'' is removed from the definition. 
In Section~\ref{C6Section} we 
show that degree $3$-critical graphs on at least $6$ vertices always contain a six-cycle. 
   In Section~\ref{RemarksSection} we make some concluding remarks and pose
several interesting open problems raised naturally by our results.
Our notation follows mostly that of \cite{BollobasModernGraphTheory}.

\section{Counterexample to Conjecture~\ref{ErdosConjecture}}\label{CounterexampleSection}
The goal of this section is to  prove Theorem~\ref{Counterexamples}. First we need some preliminary results about 1-3 trees.

Given a tree $T$, define $G(T)$ to be the graph formed from $T$ by adding two new vertices $x$ and $y$, the edge $xy$ as well as every edge between $\{x,y\}$ and the leaves of $T$. See Figure~\ref{FigureGT} for an example of a graph $G(T)$.

Notice that if $T$ is a $1$-$3$ tree then $G(T)$ is degree $3$-critical. 
In the case when $T$ is an even $1$-$3$ tree, the cycles of $G(T)$ have nice properties.
\begin{figure}[htb]
 \centering
	\begin{tikzpicture}[scale=0.8]
		\draw[rounded corners=5mm, fill=gray!20] (-0.5,-0.5) rectangle (8.5,4.5);
	
		\draw[edge] (0,0) -- (2,3) -- (2.5,2);
		\draw[edge] (1,0) -- (1.5,1) -- (2,0);
		\draw[edge] (3,0) -- (3.5,1) -- (4,0);
		\draw[edge] (1.5,1) -- (2.5,2) -- (3.5,1);
		\draw[edge] (5,0) -- (5.5,1) -- (6,0);
		\draw[edge] (7,0) -- (7.5,1) -- (8,0);
		\draw[edge] (2,3) -- (4,4) -- (5.5,1) (4,4) -- (7.5,1);
		
		\foreach \x in {0, ..., 8}
			\draw[edge] (2.5,-2) -- (\x,0) -- (5.5,-2);
		
		\draw[edge] (2.5,-2) -- (5.5,-2);

		\node[vertex] at (0,0) {};
		\node[vertex] at (1,0) {};
		\node[vertex] at (2,0) {};
		\node[vertex] at (3,0) {};
		\node[vertex] at (4,0) {};
		\node[vertex] at (5,0) {};
		\node[vertex] at (6,0) {};
		\node[vertex] at (7,0) {};
		\node[vertex] at (8,0) {};
		\node[vertex] at (1.5,1) {};
		\node[vertex] at (3.5,1) {};
		\node[vertex] at (5.5,1) {};
		\node[vertex] at (7.5,1) {};
		\node[vertex] at (2.5,2) {};
		\node[vertex] at (2,3) {};
		\node[vertex] at (4,4) {};
		
		\node[vertex, label={below:$x$}] at (2.5,-2) {};
		\node[vertex, label={below:$y$}] at (5.5,-2) {};
		
		\node at (7,3) {$T$};
	\end{tikzpicture}

\caption{The graph $G(T)$ for an even $1$-$3$ tree $T$.\label{FigureGT}}
\end{figure}
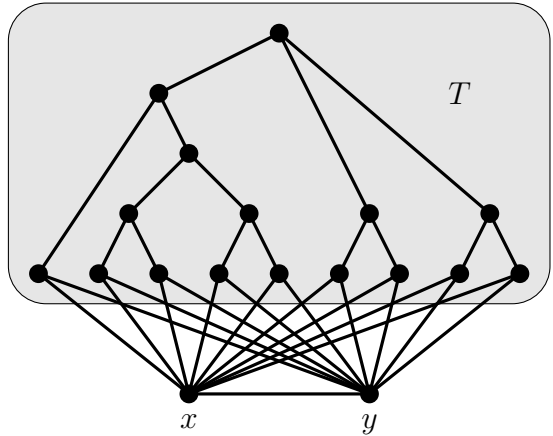

\begin{lemma}\label{CycleLengths}
Let $T$ be an even $1$-$3$-tree. Then the following hold:
\begin{enumerate}[(i)]
\item  The graph $G(T)$ contains a cycle of length $2k+1$ $\iff$ $T$ contains a leaf-leaf path of length $2k-2$.
\item  The graph $G(T)$ contains a cycle of length $2k$ $\iff$ $T$ contains two vertex-disjoint leaf-leaf paths $P_1$ and $P_2$ such that 
$e(P_1) + e(P_2) = 2k - 4$ or $T$ contains a leaf-leaf path of length $2k-2$.
\end{enumerate}
\end{lemma}
\begin{proof}
For (i), let $C$ be a $(2k+1)$-cycle in $G(T)$. Notice that since $T$ is an even tree, 
$G(T)-xy$ is bipartite. So $C$ must contain the edge $xy$ and hence
$C-x-y$ must be a leaf-leaf path of length $2k-2$ as required.  For the converse, 
notice that any path $P\subseteq T$ of length $\ell$ between leaves $u_1$ and $u_2$ 
can be turned into a cycle of length $\ell + 3$ by adding the vertices $x$ and $y$ as well as the edges $u_1x, xy, yu_2$ of $G(T)$.

For (ii), let $C$ now be a $2k$-cycle in $G(T)$.  If $|C\cap \{x,y\}|=1$ then $C-x-y$ is a leaf-leaf path in $T$ of length $2k-2$.  Now suppose that both $x,y\in V(C)$.  Notice that since $T$ is even, all leaf-leaf paths in $T$ have even length.  Therefore, all cycles containing the edge $xy$ in $G(T)$ must have odd length, and hence $C$ does not contain $xy$.  
Thus $C-x-y$ consists of two vertex-disjoint leaf-leaf paths 
$P_1, P_2 \subseteq T$ such that their lengths sum to $2k-4$, as required.  For the converse, first notice that any leaf-leaf path $P\subseteq T$ of length $\ell$ can be turned into a cycle of length $\ell + 2$ in $G(T)$ by adding the vertex $x$ and the edges between the endpoints of $P$ and $x$.  Also, any two vertex-disjoint leaf-leaf paths $P_1\subseteq T$ of length $\ell_1$ with endpoints $u_1, w_1$ and  $P_2\subseteq T$ of length $\ell_2$ with endpoints $u_2$ and $w_2$ can be turned into a cycle of length 
$\ell_1 + \ell_2 + 4$ in $G(T)$ by adding the vertices $x$ and $y$, and 
the edges $u_1x, xu_2, w_2y$, and $yw_1$ of $G(T)$.
\end{proof}

We say that a rooted binary tree $T$ is \emph{perfect} if all non-leaf vertices have two children and all root-leaf paths have the same length $d$
(or, alternatively if $|V(T)|=2^{d+1}-1$ where $d$ is the depth of $T$).  
Given a sequence of positive integers $x_1,\dots, x_n$, we define a tree $T(x_1\dots x_n)$ as follows.  First consider a path on $n$ vertices with vertex sequence $v_1,\dots,v_n$.  For each $i$ satisfying $2\leq i \leq n-1$, add a perfect rooted binary tree $T_i$ of depth $x_i-1$ with root vertex $u_i$.  For $i=1$ and $n$ add two perfect rooted binary trees each:
trees $T^{(1)}_1$ and $T^{(1)}_1$ of depths $x_1-1$ with root vertices $u^{(1)}_1$ and
$u^{(2)}_1$, respectively and trees $T^{(1)}_n$ and $T^{(1)}_n$ of depths $x_n-1$ with root vertices $u^{(1)}_n$ and $u^{(2)}_n$, respectively. 
Finally, for each $i, 2\leq i\leq n-1$, we add the edges $v_iu_i$, as well as the 
edges $v_1u^{(1)}_1$, $v_1u^{(2)}_1$, $v_nu^{(1)}_n$, and $v_nu^{(2)}_n$.
See Figure~\ref{FigureTree} for an example of a graph $G(T)$.

Notice that for any sequence  $x_1,\dots, x_n$ of positive integers, 
the tree $T(x_1\dots x_n)$ is a $1$-$3$ tree.  We will mainly be concerned with {\em odd-even sequences}, 
that is, sequences for which $x_i \equiv i \pmod{2}$ for all $i$ 
(that is, $x_i$ is even $\iff$ $i$ is even). 
It turns out that for odd-even sequences the 
leaf-leaf path length of the tree $T(x_1\dots x_n)$ are easy to characterize.
\begin{figure}
 \centering
\begin{tikzpicture}[scale=0.8]
  		\draw[edge] (0,0) -- (8,0);
  		
  		\begin{scope}[rotate=-45]
  			\draw[edge] (0,0) -- (0,-1);
  			\draw[edge] (-0.5,-2) -- (0,-1) -- (0.5,-2);
  		\end{scope}
  		
  		\draw[edge] (0,0) -- (0,-1);
  		\draw[edge] (-0.5,-2) -- (0,-1) -- (0.5,-2);
  		
  		\draw[edge] (2,0) -- (2,-1);
  		\draw[edge] (1.5,-2) -- (2,-1) -- (2.5,-2);
  		\draw[edge] (1.25,-3) -- (1.5,-2) -- (1.75,-3);
  		\draw[edge] (2.25,-3) -- (2.5,-2) -- (2.75,-3);
  		
  		\draw[edge] (4,0) -- (4,-1);
  		\draw[edge] (3.5,-2) -- (4,-1) -- (4.5,-2);
  		
  		\draw[edge] (6,0) -- (6,-1);
  		\draw[edge] (5.5,-2) -- (6,-1) -- (6.5,-2);
  		\draw[edge] (5.25,-3) -- (5.5,-2) -- (5.75,-3);
  		\draw[edge] (6.25,-3) -- (6.5,-2) -- (6.75,-3);
  		
  		\draw[edge] (8,0) -- (8,-1);
  		\draw[edge] (7.5,-2) -- (8,-1) -- (8.5,-2);
  		
  		\begin{scope}[xshift=8cm, rotate=45]
  			\draw[edge] (0,0) -- (0,-1);
  			\draw[edge] (-0.5,-2) -- (0,-1) -- (0.5,-2);
  		\end{scope}
  		
  		\node[vertex, label={above:$v_1$}] at (0,0) {};
  		\node[vertex, label={above:$v_2$}] at (2,0) {};
  		\node[vertex, label={above:$v_3$}] at (4,0) {};
  		\node[vertex, label={above:$v_4$}] at (6,0) {};
  		\node[vertex, label={above:$v_5$}] at (8,0) {};
  		
  		\begin{scope}[rotate=-45]
  			\node[vertex] at (0,-1) {};
  			\node[vertex] at (-0.5,-2) {};
  			\node[vertex] at (0.5,-2) {};
  		\end{scope}
  		
  		\node[vertex] at (0,-1) {};
  		\node[vertex] at (-0.5,-2) {};
  		\node[vertex] at (0.5,-2) {};
  		
  		\node[vertex] at (2,-1) {};
  		\node[vertex] at (1.5,-2) {};
  		\node[vertex] at (2.5,-2) {};
  		\node[vertex] at (1.25,-3) {};
  		\node[vertex] at (1.75,-3) {};
  		\node[vertex] at (2.25,-3) {};
  		\node[vertex] at (2.75,-3) {};
  		
  		\node[vertex] at (4,-1) {};
  		\node[vertex] at (3.5,-2) {};
  		\node[vertex] at (4.5,-2) {};
  		
  		\node[vertex] at (6,-1) {};
  		\node[vertex] at (5.5,-2) {};
  		\node[vertex] at (6.5,-2) {};
  		\node[vertex] at (5.25,-3) {};
  		\node[vertex] at (5.75,-3) {};
  		\node[vertex] at (6.25,-3) {};
  		\node[vertex] at (6.75,-3) {};
  		
  		\node[vertex] at (8,-1) {};
  		\node[vertex] at (7.5,-2) {};
  		\node[vertex] at (8.5,-2) {};
  		
  		\begin{scope}[xshift=8cm, rotate=45]
  			\node[vertex] at (0,-1) {};
  			\node[vertex] at (-0.5,-2) {};
  			\node[vertex] at (0.5,-2) {};
  		\end{scope}
  	\end{tikzpicture}

\caption{The $1$-$3$ tree $T(2, 3, 2, 3, 2)$. \label{FigureTree}}
\end{figure}
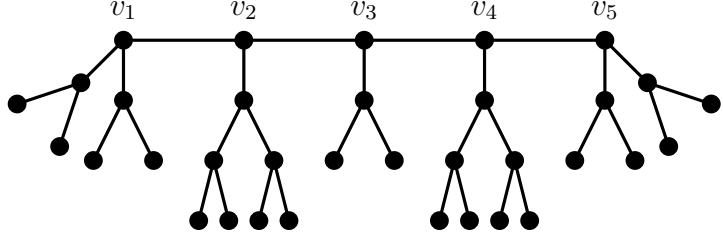

\begin{lemma}\label{PathLengths}
Let $x_1,\dots, x_n$ be an odd-even sequence. Then we have the following:
\begin{enumerate}[(i)]
\item The tree $T(x_1\dots x_n)$ contains no leaf-leaf path of odd length. In particular, $T(x_1, \ldots , x_n)$ is an even tree.
\item  For every integer $m$, $0\leq m< \max_{i=1} ^{n} x_i$, the tree
$T(x_1\dots x_n)$ contains a leaf-leaf path of length $2m$.
\item For $m=\max_{i=1} ^{n} x_i$, the tree $T(x_1\dots x_n)$ contains a leaf-leaf path 
of length $2m$ if and only if either $\max\{x_1,x_n\}=\max_{i=1} ^{n} x_i$ or  
there are two distinct integers $i$ and $j$ such that  $x_i+x_j+|i-j|=2m$.
\item For every $m> \max_{i=1} ^{n} x_i$, the tree $T(x_1\dots x_n)$ contains 
a leaf-leaf path of length $2m$ if and only if there are two distinct integers $i$ and 
$j$ such that $x_i+x_j+|i-j|=2m$.
\end{enumerate}
\end{lemma}
\begin{proof}
Leaf-leaf paths of $T(x_1, \ldots , x_n)$ can be classified based on their intersection with the
path $v_1, \ldots , v_n$. Note that this intersection is always a (potentially empty) path. 

If the intersection is empty then the path is a leaf-leaf path of a perfect binary tree of depth $x_i-1$ for some $i$, 
and hence its length is $2m$ for some $m$, $0\leq m< \max_{i=1} ^{n} x_i$.  

If the intersection is a single vertex, then this vertex must be either $v_1$ or $v_n$. 
Then the  path is a leaf-leaf path going through the root in one of the perfect binary trees 
on $V(T^{(1)}_1)\cup V(T^{(2)}_1) \cup \{v_1\}$ and $V(T^{(1)}_n)\cup V(T^{(2)}_n)\cup\{v_n\}$ of depths $x_1$ and $x_n$, respectively,
and hence its length is $2x_1$ or $2x_n$, respectively.

If the intersection is a segment $v_i, \ldots , v_j$ for some $1\leq i < j \leq n$, then the path has length $x_i + j-i + x_j$. 
This implies the ``only if '' part of (iii) and (iv). 
Note also that all these paths have even length ($x_i+j-i+x_j$ is even because $(x_1, \ldots , x_n)$ is an odd-even sequence), and so (i) holds. 

For (ii) and the ``if'' part of (iii) and (iv) one must only note that a perfect tree of depth $d$ contains a leaf-leaf path of every even length
$0, 2, \dots, 2d$ and hence all leaf-leaf path-lengths given by the classification can actually be realized.
\end{proof}

We now produce a sequence of integers $(x_n)_{n=1}^{\infty}$ such that for every $n$, the tree $T(x_1\dots x_n)$ will not have leaf-leaf paths of length $20$.

\subsection{$k$-avoiding sequences}

We will be concerned with two-sided sequences $\seq{a}$ of positive integers. Again, we say that such a sequence is an \emph{odd-even sequence} if $a_i \equiv i \pmod{2}$ for all $i \in \Z$.

\begin{definition}
Let $k$ be a positive even integer. A two-sided sequence $\seq{x}$ of positive integers is called \emph{$k$-avoiding} if $a_i \leq k / 2$ for all $i \in \Z$ and if for every $i, j \in \Z$, $i \neq j$, we have $a_i + a_j + \abs{i - j} \neq k$.
\end{definition}

In order to check if an odd-even sequence $\seq{a}$ with $a_i \leq k/2$ for all $i \in \Z$ is $k$-avoiding, consider the graph $\set{(i, a_i)}{i \in \Z}$ of the sequence. Call a point $(x, y)\in \Z\times [ 1,k/2]$ 
\emph{in conflict} with another point $(z,w)\in \Z\times[1,k/2]$, $(z, w) \neq (x, y)$, if $y + w + \abs{x - z} = k$. Notice that the points 
$(x,y)$ in conflict with a fixed point $(c, d)$ lie on the two diagonal lines $y = -x + (k + c - d)$ and $y = x + (k - c - d)$. 
Since being in conflict is a symmetric relation we can say that we {\em blame} a conflict on the point with lower first coordinate 
(the first coordinates of points in conflict cannot be equal). 
Then the points $(x,y) \in \Z\times [1,k/2]$, whose conflicts with $(c, d)$ are blamed on $(c, d)$
lie on the single line $y = -x + (k + c - d)$. Indeed, the first coordinates of a point 
$(x,y)$ on the other diagonal line is $x=y-k+c+d \leq k/2 - k + c + k/2 $ at most $c$, hence these conflicts are not blamed on $(c,d)$.
We define the \emph{fault line} of the point $(c,d)$ to be the line $y = -x + (k + c - d)$. 
From the above discussion we obtain the following proposition.
\begin{proposition}\label{FaultLine}
A sequence $\seq{a}$ is $k$-avoiding if, and only if, there do not exist two distinct indices $i$ and $j$ such that $(i,a_i)$ lies on the fault line of $(j,a_j)$.
\end{proposition}
It is useful to note that all the points on the line $y = x + b$ have the same fault line $y = -x + b + k$.

\begin{theorem} \label{thm:20-avoiding}
There is a $20$-avoiding odd-even sequence.
\end{theorem}

\begin{proof}
Let $\seq{a}$ be the periodic sequence of period $24$ consisting of repetitions of 
\[
	\dots, 1, 2, 1, 4, 3, 2, 7, 6, 5, 6, 7, 2, 3, 4, 1, 2, 1, 8, 9, 6, 5, 6, 9, 8, \dots.
\]
We claim $\seq{a}$ is a $20$-avoiding odd-even sequence. It is clearly an odd-even sequence, and $a_i \leq 10 = 20 / 2$ for all $i \in \Z$. 
We prove that it is $20$-avoiding by showing that in the graph of this sequence, no point lies on the fault line of another point. Then Proposition~\ref{FaultLine} implies the theorem.

Figure~\ref{fig:20avoiding} is a snapshot of two periods of the graph. The points on the graph are black circles, and the fault lines are drawn in red. Note that points on a line $\ell$ parallel to the line ``$x = y$'' have  the same fault line, and that this fault line crosses $\ell$ when the second coordinate is $10$.

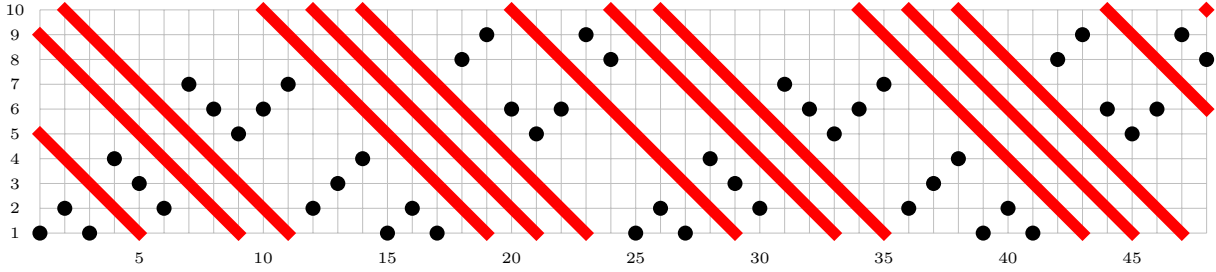
\begin{figure}[htb]
\begin{center}
\begin{tikzpicture}[scale=0.33]

\foreach \y in {1, ..., 10}
	\node at (0, \y) {\tiny \y};
\foreach \x in {5, 10,..., 45}
	\node at ({\x}, 0) {\tiny {\x}};

\draw[step=1, very thin, gray!50] (1,1) grid (48,10);

	  	


\draw[fault line] (1,5) -- (5,1);
\draw[fault line] (1,9) -- (9,1);
\draw[fault line] (2,10) -- (11,1);
\draw[fault line] (10,10) -- (19,1);
\draw[fault line] (12,10) -- (21,1);
\draw[fault line] (14,10) -- (23,1);
\draw[fault line] (20,10) -- (29,1);
\draw[fault line] (24,10) -- (33,1);
\draw[fault line] (26,10) -- (35,1);
\draw[fault line] (34,10) -- (43,1);
\draw[fault line] (36,10) -- (45,1);
\draw[fault line] (38,10) -- (47,1);
\draw[fault line] (44,10) -- (48,6);
\draw[fault line] (48,10) -- (47.99,10.01);


\node[oo] at (1,1) {};
\node[oo] at (2,2) {};
\node[oo] at (3,1) {};
\node[oo] at (4,4) {};
\node[oo] at (5,3) {};
\node[oo] at (6,2) {};
\node[oo] at (7,7) {};
\node[oo] at (8,6) {};
\node[oo] at (9,5) {};
\node[oo] at (10,6) {};
\node[oo] at (11,7) {};
\node[oo] at (12,2) {};
\node[oo] at (13,3) {};
\node[oo] at (14,4) {};
\node[oo] at (15,1) {};
\node[oo] at (16,2) {};
\node[oo] at (17,1) {};
\node[oo] at (18,8) {};
\node[oo] at (19,9) {};
\node[oo] at (20,6) {};
\node[oo] at (21,5) {};
\node[oo] at (22,6) {};
\node[oo] at (23,9) {};
\node[oo] at (24,8) {};
\node[oo] at (25,1) {};
\node[oo] at (26,2) {};
\node[oo] at (27,1) {};
\node[oo] at (28,4) {};
\node[oo] at (29,3) {};
\node[oo] at (30,2) {};
\node[oo] at (31,7) {};
\node[oo] at (32,6) {};
\node[oo] at (33,5) {};
\node[oo] at (34,6) {};
\node[oo] at (35,7) {};
\node[oo] at (36,2) {};
\node[oo] at (37,3) {};
\node[oo] at (38,4) {};
\node[oo] at (39,1) {};
\node[oo] at (40,2) {};
\node[oo] at (41,1) {};
\node[oo] at (42,8) {};
\node[oo] at (43,9) {};
\node[oo] at (44,6) {};
\node[oo] at (45,5) {};
\node[oo] at (46,6) {};
\node[oo] at (47,9) {};
\node[oo] at (48,8) {};
\end{tikzpicture}
\caption{A snapshot of the graph of a periodic $20$-avoiding odd-even sequence.} \label{fig:20avoiding}
\end{center}
\end{figure}

From the picture we see that no point of the sequence lies on a fault line of another point, implying that $\seq{a}$ is indeed $20$-avoiding.
\end{proof}


We are now ready to prove part (ii) of Theorem~\ref{TreeLengths} and 
Theorem~\ref{Counterexamples}.
\begin{proof}[Proof of part (ii) of Theorem~\ref{TreeLengths}]
Let $x_1,\dots,x_n$ be the first $n$ terms (starting at $1$) of the $20$-avoiding sequence produced by Theorem~\ref{thm:20-avoiding}. The tree $T_n=T(x_1\dots x_n)$ is a 1-3-tree 
for any sequence $(x_1, \ldots , x_n)$ by construction.  
Since $(x_1, \ldots , x_n)$ is an odd-even sequence, $T_n$ is also an 
even tree by part (i) of Lemma~\ref{PathLengths}.
The tree $T_n$
contains no leaf-leaf paths of length $20$, since $20 > 2\cdot \max x_i  =18$ and part (iv) of Lemma~\ref{PathLengths} tells us that a leaf-leaf path
of length $20$ exists only if there are distinct $i$ and $j$ such that $x_i+x_j +|i-j| =20$,
which is not case since $x_1, \ldots , x_n$ is $20$-avoiding.
\end{proof}
\begin{proof}[Proof of Theorem~\ref{Counterexamples}]
We let $G_n=G(T_n)$ be the graph constructed from the tree $T_n$ given by part (ii) 
of Theorem~\ref{TreeLengths}.
Since $T_n$ is a 1-3-tree, the graph $G_n$ is degree $3$-critical, as required. 
Since $T_n$ is an even 1-3-tree, we can use part (i) of Lemma~\ref{CycleLengths} and 
the fact that $T_n$ does not 
contain a leaf-leaf path of length $20$ to conclude that  
$G_n$ contains no cycle of length $23$.  
\end{proof}

\section{Possible leaf-leaf path lengths in even $1$-$3$ trees}\label{TreeLengthsSection}
In this section we prove part (i) of Theorem~\ref{TreeLengths}.  We first need a lemma about possible lengths of leaf-leaf paths in binary trees which have no short root-leaf paths.
\begin{lemma}\label{DeepTree}
Let $T$ be an even rooted binary tree and let $m$ be the length of its shortest root-leaf path.  
Then $T$ contains leaf-leaf paths of lengths $0,2,4,\dots,2m$.
\end{lemma}
\begin{proof}
The proof is by induction on $|V(T)|$.  The statement is certainly true for $|V(T)|=1$.  
Let now $|V(T)| >1$ and let $x$ and $y$ be the children of of the root $r$. 

Suppose first that in one of the subtrees $T_x$ and $T_y \subseteq T$, rooted at $x$ and $y$, respectively, the
shortest root-leaf path is of length $m$ as well. 
In this case we can apply induction to this subtree and find in it a leaf-leaf paths of all
length $0,2,\ldots , 2m$.  The leaf-leaf path of the subtree are of course leaf-leaf paths of $T$, so we are done in this case.  

Otherwise, the length of the shortest root-leaf path of both subtrees $T_x$ and $T_y$ are $m-1$ 
(the subtrees cannot contain a shorter root-leaf path, because $T$ itself does not contain a root-leaf path shorter than $m$). 
Then by induction there are leaf-leaf path of all length $0,2, \ldots , 2m-2$ in both of these subtrees and hence also in $T$.
To construct a leaf-leaf path of length $2m$ in $T$
let $P_x$ be a path between $x$ and a leaf of $T$ of length $m-1$, and $P_y$ be a path between $y$ and a leaf of $T$ of length 
$m-1$. Then the path $P_x+r+P_y$ formed by joining $P_x$ and $P_y$ to $r$ using the edges $rx$ and $ry$ is a 
leaf-leaf path in $T$ of length $2m$.
\end{proof}

The following proposition shows that finding which leaf-leaf paths lengths always occur in sufficiently large trees is equivalent to finding the $k$ for which $k$-avoiding sequences exist.
\begin{proposition}\label{TreeSequenceEquivalence}
Let $m$ be a positive integer.  The following are equivalent.
\begin{enumerate}[(i)]
\item There is an integer $N_0(m)$ 
such that every even 1-3-tree of order at least $N_0(m)$ contains a leaf-leaf path of length $2m$.
\item There exists no $2m$-avoiding odd-even sequence $(x_n)_{n \in \Z}$. 
\end{enumerate}
\end{proposition}
\begin{proof}
Let us assume first that (i) holds with integer $N_0(m)=N_0$.  
Let  $(x_n)_{n=1} ^{\infty}$ be an arbitrary odd-even sequence such that $\max_{i=1} ^{n} x_i\leq m$.  
Notice that since 
$x_i$ is even if and only if $i$ is even, there are infinitely many indices $a$ for which $x_a<\max_{i=1} ^{n} x_i$.  
Therefore we can choose two indices $a$ and $b$ such that $a-b\geq N_0$ and $x_a$, $x_b < m$.  
Then by parts (iii) and (iv) of Lemma~\ref{PathLengths}, the tree $T(x_a\dots x_b)$ has a leaf-leaf path of length 
$2m$ if and only if there are two distinct indices $i$ and $j$ such that $x_i+x_j+|i-j|=2m$ holds.  
On the other hand notice that 
by part (i) of Lemma~\ref{PathLengths}, $T(x_a\dots x_b)$ is an even 1-3 tree and hence, since its order is at least $N_0$, 
does have a leaf-leaf path of length $2m$. That is, there do exist indices $i\neq j$ 
such that $x_i+x_j+|i-j|=2m$ holds, implying that $(x_n)_{n=1} ^{\infty}$ is not $2m$-avoiding.

Now assume that (ii) holds. Let us define $N_0(m)=N_0=\frac{3}{2}\cdot2^{N_1/2} -1$, where $N_1=m^{2m}+2m$.  
Let $T$ be an arbitrary even $1$-$3$ tree of order at least $N_0$. We will show that
$T$ contains a leaf-leaf path of length $2m$. 
Since $T$ is a tree of maximum degree at most $3$ on $N_0$ vertices it 
must contain a path $v_1, v_2, \dots, v_{N_1}$ with $N_1$ vertices. 
Let $T_i$ be the subtree of $T$ consisting of the connected component of 
$T-v_{i+1}-v_{i-1}$ 
containing $v_i$ and let $x_i$ be the length of the shortest path from $v_i$ 
to a leaf of $T_i$. Note that $(x_i)_{i=1}^{N_1}$ is an odd-even sequence, because $T$ is an
even tree.

Suppose first that we have $m<\max_{i=1} ^{N_1} x_i$.  Choose an index $i$ such that $x_i>m$ holds and let $T'=T_i-v_i$.  
Then $T'$ is a binary tree rooted at the neighbour of $v_i$, with no root-leaf paths shorter than $m$, so
Lemma~\ref{DeepTree} gives us a leaf-leaf path of length $2m$. 

Suppose now that we have $m\geq\max_{i=1} ^{N_1} x_i$.  
Since $N_1> m^{2m} + 2m - 1$, the Pigeonhole Principle implies that there must be 
indices $a< b$ such that 
$x_a=x_b, x_{a+1}=x_{b+1}, \dots, x_{a+2m - 1}=x_{b+2m - 1}$ all hold.  
Consider now the infinite periodic sequence 
$$\dots, x_a, x_{a+1}, \dots, x_{b-1}, x_a, x_{a+1}, \dots, x_{b-1},  x_a,  \dots,$$
denoted by $(y_i)_{i\in \Z}$.
This is an odd-even sequence as the sequence $(x_i)_{i=1}^{N_1}$ was odd-even.
By our assumption $(y_i)_{i\in \Z}$ is not $2m$-avoiding. But 
$m\geq\max_{i=1} ^{n} x_i=\max_{i=1} ^{n} y_i$, so there must be indices 
$i\neq j$ such that $y_i+y_j+|i-j|=2m$.
Since the sequence is positive we must have $|i-j|< 2m$ and by periodicity
we can assume that $a\leq i < j \leq b+2m-1$. The way we chose $a$ and $b$ ensures
that $x_i=y_i$ for every $i$ between $a$ and $b+2m-1$, so we also have 
$x_i+x_j+|i-j|=2m$.
We can now find a leaf-leaf path in $T$ of length $2m=x_i+x_j+|i-j|$ by 
concatenating a shortest path from $v_i$ to a leaf of $T_i$,
the path between $v_i$ and $v_j$ and a 
shortest  path from $v_j$ to a leaf of $T_j$.
\end{proof}

We now proceed to prove part (i) of Theorem~\ref{TreeLengths}.  We do this by showing that part (ii) of Proposition~\ref{TreeSequenceEquivalence} holds for $m\leq 9$.

\begin{theorem} \label{thm:18-avoiding}
There is no $18$-avoiding odd-even sequence.
\end{theorem}
\begin{proof}
Consider an odd-even sequence $\seq{a}$ with $a_i \leq 9$ for all $i \in \Z$. Assume that it is $18$-avoiding. As in the proof of Theorem~\ref{thm:20-avoiding}, we will consider the graph of $\seq{a}$ and consider fault lines. In this case, the fault line of a point $(c, d)$ is the line $y = -x + (18 + c - d)$. Since $\seq{a}$ is $18$-avoiding, Proposition~\ref{FaultLine} implies that no point of the graph lies on
the fault line of another point of the graph.
Notice however, that  a point of the form $(x, 9)$, which by
definition lies on its own fault line, is not itself a barrier to a sequence being $18$-avoiding.

We start with some lemmas about configurations of fault lines that lead to contradictions. We will actually deal with a slight generalization of fault lines, which we call \emph{excluded lines}. An excluded line is defined to be any line of  
the form $y=-x+b$  with $b$ even, 
that does not contain a point in the graph of $\seq{a}$, except possibly
the point with second coordinate $9$. Since $\seq{a}$ is an odd-even sequence, 
for any point $(i,a_i)$ in the graph of the sequence, the integer $18+i-a_i$ is even.
Hence every fault line of the sequence is also an excluded line.

In the following discussion lines of slope $-1$ whose $y$-intercepts differ by exactly $2$ are called \emph{consecutive}. We start with a trivial observation. 

\begin{lemma} \label{lem:4fault}
There cannot be four consecutive excluded lines for $\seq{a}$.
\end{lemma}

\begin{proof}
If there were four 
excluded lines $y = -x + b$, $y = -x + b + 2$, $y = -x + b + 4$, and $y = -x + b + 6$, where $b$ is even, then all of the points with even $y$-coordinate at most $8$ on the line $x = b - 2$ are on one of these lines. Hence $(b - 2, a_{b - 2})$ would be on an excluded line, 
a contradiction.

\begin{center}
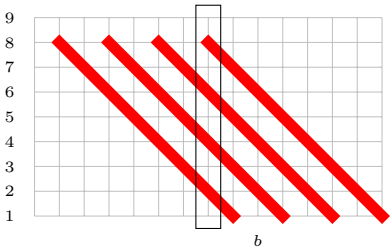

\begin{tikzpicture}[scale=0.33]

\draw[step=1, very thin, gray!50] (1,1) grid (15,9);

\foreach \y in {1, ..., 9}
	\node at (0, \y) {\tiny \y};
\node at (10, 0) {\tiny $b$};


\draw[fault line] (2,8) -- (9,1);
\draw[fault line] (4,8) -- (11,1);
\draw[fault line] (6,8) -- (13,1);
\draw[fault line] (8,8) -- (15,1);


\draw (7.5,0.5) rectangle (8.5,9.5);
\end{tikzpicture}
\captionof{figure}{Four consecutive excluded lines and the contradiction they give. 
}
\end{center}
\end{proof}

This easily leads to the next lemma.

\begin{lemma} \label{lem:3fault}
There cannot be three consecutive excluded lines for $\seq{a}$.
\end{lemma}

\begin{proof}
If there were three consecutive excluded lines $y = -x + b$, $y = -x + b + 2$, and $y = -x + b + 4$, where $b$ is even, then $a_{b - 4}$ must be equal to $2$ as all the other even values at most $8$ would put $(b - 4, a_{b - 4})$ on one of the three lines. Similarly, we must have $a_{b - 2} = 8$, and hence we have fault lines $y = -x + b + 8$ and $y = -x + b + 12$. This forces $a_{b - 1} = 7$, giving also the fault line $y = -x + b + 10$. Now $a_{b - 3}$ can only be $1$ or $9$ to avoid the original three fault lines, but it clearly cannot be $9$, since that would put $(b - 2, 8)$ on its fault line. But if $a_{b - 3} = 1$, then its fault line is $y = -x + b + 14$, and there would be $4$ consecutive fault lines $y = -x + b + \lset{8, 10, 12, 14}$, contradicting Lemma~\ref{lem:4fault}.

\begin{center}
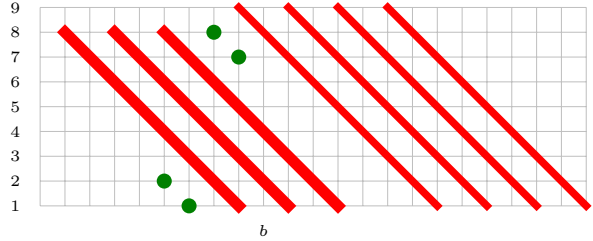

\begin{tikzpicture}[scale=0.33]

\draw[step=1, very thin, gray!50] (1,1) grid (23,9);

\foreach \y in {1, ..., 9}
	\node at (0, \y) {\tiny \y};
\node at (10, 0) {\tiny $b$};


\draw[fault line] (2,8) -- (9,1);
\draw[fault line] (4,8) -- (11,1);
\draw[fault line] (6,8) -- (13,1);
\draw[fault line thin] (9,9) -- (17,1);
\draw[fault line thin] (11,9) -- (19,1);
\draw[fault line thin] (13,9) -- (21,1);
\draw[fault line thin] (15,9) -- (23,1);


\node[ooo] at (6,2) {};
\node[ooo] at (7,1) {};
\node[ooo] at (8,8) {};
\node[ooo] at (9,7) {};
\end{tikzpicture}
\captionof{figure}{Three consecutive excluded lines and the contradiction they give.}
\end{center}
\end{proof}

A few more lemmas of this sort will be useful for the proof.

\begin{lemma} \label{lem:21fault}
There cannot be three excluded lines of the form $y = -x + b$, $y = -x + b + 2$, and $y = -x + b + 6$ (with $b$ even).
\end{lemma}

\begin{proof}
If this were the case, then this would force $a_{b - 2} = 6$, which results in the fault line $y = -x + b + 10$. If $a_{b - 1} = 9$, then there would be
three consecutive excluded lines $y = -x + b + \lset{6, 8, 10}$, which would contradict Lemma~\ref{lem:3fault}. This forces $a_{b - 1} = 5$, which results in the fault line $y = -x + b + 12$. Similarly, in order to avoid a third consecutive fault line $y = -x + b + 14$, we must have $a_{b + 2} = 2$ and $a_{b + 5} = 3$, resulting in the fault lines $y = -x + b + 18$ and $y = -x + b + 20$, respectively. This leaves us with no valid choices for $a_{b + 6}$, since a value of $2$ would create a third consecutive fault line $y = -x + b + 22$, a value of $8$ would create a third consecutive fault line $y = -x + b + 16$, and a value of $4$ or $6$ would put $(b + 6, a_{b + 6})$ on the fault line of a previous point. Therefore, this configuration cannot occur.

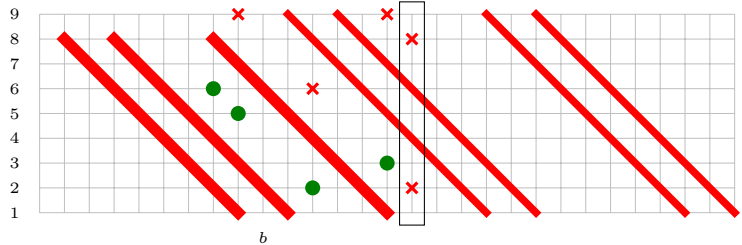
\begin{center}
\begin{tikzpicture}[scale=0.33]

\draw[step=1, very thin, gray!50] (1,1) grid (29,9);

\foreach \y in {1, ..., 9}
	\node at (0, \y) {\tiny \y};
\node at (10, 0) {\tiny $b$};


\draw[fault line] (2,8) -- (9,1);
\draw[fault line] (4,8) -- (11,1);
\draw[fault line] (8,8) -- (15,1);
\draw[fault line thin] (11,9) -- (19,1);
\draw[fault line thin] (13,9) -- (21,1);
\draw[fault line thin] (19,9) -- (27,1);
\draw[fault line thin] (21,9) -- (29,1);


\node[ooo] at (8,6) {};
\node[ooo] at (9,5) {};
\node[ooo] at (12,2) {};
\node[ooo] at (15,3) {};

\node[not] at (9,9) {};
\node[not] at (12,6) {};
\node[not] at (15,9) {};
\node[not] at (16,2) {};
\node[not] at (16,8) {};

\draw (15.5,0.5) rectangle (16.5,9.5);
\end{tikzpicture}
\captionof{figure}{A configuration of three excluded lines and the contradiction they give.}
\end{center}
\end{proof}

\begin{lemma} \label{lem:12fault}
There cannot be three excluded lines of the form $y = -x + b$, $y = -x + b + 4$, and $y = -x + b + 6$ (with $b$ even).
\end{lemma}

\begin{proof}
If this were the case, then this would force $a_{b - 2} = 4$, which results in the fault line $y = -x + b + 12$. If $a_{b - 1} = 9$, there would be three
 consecutive excluded lines $y=-x +b +\{ 4,6,8\}$, contradicting Lemma~\ref{lem:3fault}. So we must have $a_{b - 1} = 3$, resulting in the fault line $y = -x + b + 14$ (the other values of $a_{b - 1}$ would put $(b - 1, a_{b - 1})$ on an excluded line). If $a_b = 8$, then we would have the configuration of fault lines $y = -x + b + \lset{4, 6, 10}$ forbidden by Lemma~\ref{lem:21fault}, so we must have $a_b = 2$, resulting in the fault line $y = -x + b + 16$. But then we have the three consecutive fault lines $y = -x + b + \lset{12, 14, 16}$, also a contradiction.

\begin{center}
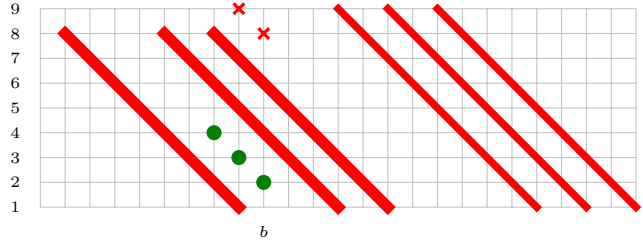

\begin{tikzpicture}[scale=0.33]

\draw[step=1, very thin, gray!50] (1,1) grid (25,9);

\foreach \y in {1, ..., 9}
	\node at (0, \y) {\tiny \y};
\node at (10, 0) {\tiny $b$};


\draw[fault line] (2,8) -- (9,1);
\draw[fault line] (6,8) -- (13,1);
\draw[fault line] (8,8) -- (15,1);
\draw[fault line thin] (13,9) -- (21,1);
\draw[fault line thin] (15,9) -- (23,1);
\draw[fault line thin] (17,9) -- (25,1);


\node[ooo] at (8,4) {};
\node[ooo] at (9,3) {};
\node[ooo] at (10,2) {};

\node[not] at (9,9) {};
\node[not] at (10,8) {};
\end{tikzpicture}
\captionof{figure}{A configuration of three excluded lines and the contradiction they give.}
\end{center}
\end{proof}

All previous lemmas pave way for our final technical lemma:

\begin{lemma} \label{lem:2fault}
There cannot be $2$ consecutive  excluded lines for $\seq{a}$.
\end{lemma}

\begin{proof}
Suppose there were two consecutive excluded lines $y = -x + b$ and $y = -x + b + 2$ for some even $b$. Consider the possible values for $a_{b - 8}$. It cannot be $8$, since this is on the excluded line $y = -x + b$. It cannot be $6$, as this this would create a third consecutive 
excluded line $y = -x + b + 4$. It also cannot be $4$, because this would create the fault line $y = -x + b + 6$, contradicting Lemma~\ref{lem:21fault}. Thus, we must have $a_{b - 8} = 2$, which means we have the fault line $y = -x + b + 8$. We also must have $a_{b - 4} = 2$, since values 
$4$ or $6$ would put a point of the graph on one of the excluded lines, and value $8$ would yield the fault line $y = -x + b + 6$, contradicting Lemma~\ref{lem:21fault}. Thus, we also have the fault line $y = -x + b + 12$.

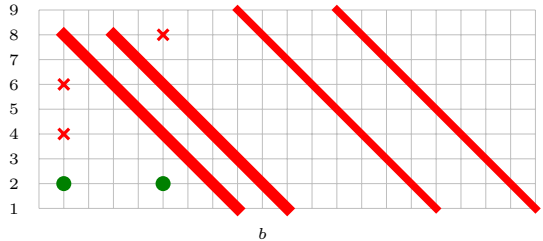
\begin{center}
\begin{tikzpicture}[scale=0.33]

\draw[step=1, very thin, gray!50] (1,1) grid (21,9);

\foreach \y in {1, ..., 9}
	\node at (0, \y) {\tiny \y};
\node at (10, 0) {\tiny $b$};


\draw[fault line] (2,8) -- (9,1);
\draw[fault line] (4,8) -- (11,1);
\draw[fault line thin] (9,9) -- (17,1);
\draw[fault line thin] (13,9) -- (21,1);


\node[ooo] at (2,2) {};
\node[ooo] at (6,2) {};

\node[not] at (2,4) {};
\node[not] at (2,6) {};
\node[not] at (6,8) {};
\end{tikzpicture}
\captionof{figure}{What two consecutive excluded lines can be reasoned to imply.}
\end{center}

Now consider $a_{b + 1}$. It cannot be $1$ or $7$, since these would put a point of the graph on an excluded line. It cannot be $9$, otherwise it would create the fault line $y = -x + b + 10$, and we would have three consecutive fault lines $y = -x + b + \lset{8, 10, 12}$ contradicting 
Lemma~\ref{lem:3fault}. It cannot be $5$, for if it were, there would be the three fault lines $y = -x + b + \lset{8, 12, 14}$ in contradiction with Lemma~\ref{lem:12fault}. Hence we have $a_{b + 1} = 3$ and the fault line $y = -x + b + 16$. Similarly, we must have $a_{b + 3} = 1$, since $5$ or $9$ put it on a fault line, $7$ would create three consecutive fault lines $y = -x + b + \lset{12, 14, 16}$, and $3$ would create the configuration of fault lines $y = -x + b + \lset{12, 16, 18}$ forbidden by Lemma~\ref{lem:12fault}. Therefore, there is also the fault line $y = -x + b + 20$. But now every possible value for $a_{b + 7}$ leads to a contradiction. If it is $1$, $5$, or $9$, then it is on a fault line. If it is $3$ or $7$, it creates a fault line resulting in a configuration forbidden by Lemmas~\ref{lem:12fault} and~\ref{lem:3fault}, respectively. 
Therefore, we cannot have two consecutive excluded lines.

\begin{center}
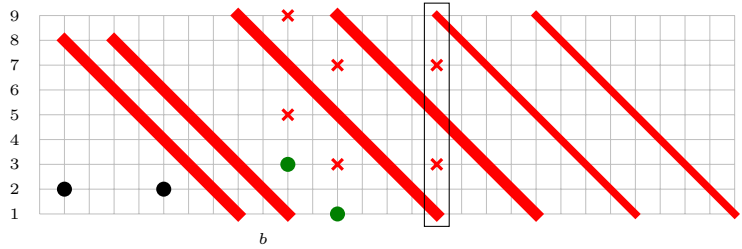

\begin{tikzpicture}[scale=0.33]

\draw[step=1, very thin, gray!50] (1,1) grid (29,9);

\foreach \y in {1, ..., 9}
	\node at (0, \y) {\tiny \y};
\node at (10, 0) {\tiny $b$};


\draw[fault line] (2,8) -- (9,1);
\draw[fault line] (4,8) -- (11,1);
\draw[fault line] (9,9) -- (17,1);
\draw[fault line] (13,9) -- (21,1);
\draw[fault line thin] (17,9) -- (25,1);
\draw[fault line thin] (21,9) -- (29,1);


\node[oo] at (2,2) {};
\node[oo] at (6,2) {};
\node[ooo] at (11,3) {};
\node[ooo] at (13,1) {};

\node[not] at (11,5) {};
\node[not] at (11,9) {};
\node[not] at (13,3) {};
\node[not] at (13,7) {};
\node[not] at (17,3) {};
\node[not] at (17,7) {};

\draw (16.5,0.5) rectangle (17.5,9.5);
\end{tikzpicture}
\captionof{figure}{The contradiction reached from two consecutive fault or excluded lines.}
\end{center}
\end{proof}

Excluded lines by definition have slope $-1$.
To finish the proof pf Theorem~\ref{thm:18-avoiding} we extend the notion of excluded line to those lines 
$y=x+b$ with even $b$, which do not contain any point of the graph except possibly the point 
$(9-b,9)$. We call these the {\em orthogonal excluded line}s of the sequence $\seq{a}$. 
The conclusion of Lemma~\ref{lem:2fault} also holds for orthogonal extended lines:
there cannot be two consecutive ones. Indeed, $y=x+b$ is an orthogonal extended line of 
the $18$-avoiding odd-even sequence $\seq{a}$ if and only if $y=-x -b$ is an extended line of 
 the $18$-avoiding odd-even sequence $\rev{a}$, so we can apply Lemma~\ref{lem:2fault} for 
 $\rev{a}$.
 
Another useful observation is that the line $y=-x+b$ is the fault line of  exactly those points that are on the line $y=x+18-b$. Hence 
if $y=-x+b$ contains a point of the graph (say, it is not excluded), then $y=x+ 18 -b$ must be an orthogonal excluded line. 
Using this observation  for $\rev{a}$ one can also obtain that if the orthogonal line 
$y=x+b$ contains a point of the graph of $\seq{a}$  
(say, it is not excluded), then $y=-x-18-b$ must be an excluded line for $\seq{a}$. 

Let us now assume that there exists an $18$-avoiding sequence $\seq{a}$ and 
let $y=-x+b$ be a fault line of it for some even $b$. By the above we can make a 
sequence of conclusions. The lines  $y=-x +b \pm 2$ are not excluded by 
Lemma~\ref{lem:2fault}. Then $y= x + 18 - b \pm 2$ must be orthogonal excluded lines.
Then $y = x + 18 - b \pm 4$ are not excluded  by the adaptation of 
Lemma~\ref{lem:2fault} for orthogonal lines. Then $y= - x + b \pm 4$ must  
be excluded lines. Again by Lemma~\ref{lem:2fault} the lines $y= -x +b \pm 6$ 
are not excluded and hence the orthogonal lines $y= x + 18 - b \pm 6$ 
must be excluded.  This implies that $y= x + 18 - b \pm 8$ are not orthogonal 
excluded lines by the adaptation of Lemma~\ref{lem:2fault} and 
$y= - x + b \pm 8$ are excluded lines.

What can now be the value of $a_{b-9}$? It must be odd as $b$ 
is even and $\seq{a}$ is an odd-even sequence. The line 
$y= x + 18 - b - 2$ being excluded shows it cannot be $7$, 
$y= - x + b - 4$ being excluded shows that it cannot $5$,
$y= x + 18 - b - 6$ being excluded shows it cannot be $3$,
$y= x + b - 8$ being  excluded shows it cannot be $1$.
The line $y=-x+b$ is a fault line of $\seq{a}$ so in principle $(a_{b-9}, 9)$ could be on it.
However then, the orthogonal line $x+18-b$ should also be excluded, 
meaning that together with $y= x + 18 - b \pm 2$ they would represent
three consecutive orthogonal excluded lines, a contradiction.

\end{proof}

To complete the proof of Theorem~\ref{TreeLengths} we need the following 
little proposition.

\begin{proposition}\label{SteppingUp}
Let $k$ be a positive even integer. If there is a $k$-avoiding odd-even sequence, then there is a $(k + 2\ell)$-avoiding odd-even sequence for every $\ell \in \Z_{\geq 0}$.
\end{proposition}

\begin{proof}
If $\seq{a}$ is a $k$-avoiding odd-even sequence, then define the sequence $\seq{b}$ by
\[
	b_i = a_{i + \ell} + \ell
\]
for all $i \in \Z$. We claim that $\seq{b}$ is a $(k + 2 \ell)$-avoiding odd-even sequence.

It is clearly an odd-even sequence as $b_i = a_{i + \ell} + \ell \equiv i + 2 \ell \equiv i \pmod{2}$ for all $i \in \Z$. Also, $b_i = a_{i + \ell} + \ell \leq k / 2 + \ell = (k + 2 \ell) / 2$ for all $i \in \Z$. Suppose there were $i, j \in \Z$ with $i < j$ such that $b_i + b_j - i + j = k + 2 \ell$. Then we would have $a_{i + \ell} + \ell + a_{j + \ell} + \ell - i + j = k + 2 \ell$. But this implies $a_{i + \ell} + a_{j + \ell} - (i + \ell) + (j + \ell) = k$, which contradicts the fact that $\seq{a}$ is $k$-avoiding.
\end{proof}

\begin{proof}[Proof of Theorem~\ref{TreeLengths}~(i)]
Let $m \leq 9$ be a positive integer. We claim that there is no $2m$-avoiding 
odd-even sequence. Indeed, 
otherwise our previous proposition implied that there is also an $18$-avoiding odd-even sequence, which contradicts Theorem~\ref{thm:18-avoiding}.
Now by Proposition~\ref{TreeSequenceEquivalence}, there is an integer 
$N_0(m)$ such that every even $1$-$3$ tree of order at least $N_0(m)$ contains 
a leaf-leaf path of length $2m$, which is exactly the statement of 
part~(i) of Theorem~\ref{TreeLengths}.
\end{proof}

\section{Characterization of graphs with no subgraphs of minimum degree 3}\label{NonInducedSection}
Let ${\cal G}$ denote the family of graphs $G$ 
with $2|G|-2$ edges  and no proper (not necessarily induced) subgraphs with minimum degree $3$. In this section we characterize the members of
${\cal G}$ and deduce Theorem~\ref{NonInduced} as a corollary.

A wheel $W_n$ is an $n$-vertex graph with vertices $c$, and $w_1, \dots, w_{n-1}$ with edges $cw_i$ and $w_iw_{i+1\pmod{n-1}}$ for $i=1, \dots, n-1$.  The vertex $c$ will be called the \emph{centre} of the wheel and the vertices  $w_1, \dots, w_{n-1}$ will be called the \emph{outside vertices} of $W_n$.
For $n\geq 4$,  Let $H_n$ be the graph on $n$ vertices called $x$, $y$, and $v_1, \dots, v_{n-2}$ formed by the edges $v_iv_{i+1}$ for $i\in\{1, \dots, n-3\}$, $xv_{i}$ for $i\in\{1, \dots, n-2\}$, $yv_1$, and $y v_{n-2}$.  We call $x$ and $y$ the \emph{connectors} of $H_n$ and  $v_1, \dots, v_{n-2}$ the \emph{internal vertices} of $H_n$.   
Note that the roles of the connectors are not symmetric;
the letter $y$ will always denote one with degree two. 
See Figure~\ref{HnFigure} for a picture of the graph $H_7$.
\begin{figure}
  \centering
  	\begin{tikzpicture}[scale=0.8]
		\draw[edge] (-2,0) -- (0,2);
		\draw[edge] (-2,0) -- (0,-2);
		\draw[edge] (0,2) -- (0,1);
		\draw[edge] (0,2) -- (2,0);
		\draw[edge] (0,1) -- (0,0);
		\draw[edge] (0,1) -- (2,0);
		\draw[edge] (0,0) -- (0,-1);
		\draw[edge] (0,0) -- (2,0);
		\draw[edge] (0,-1) -- (0,-2);
		\draw[edge] (0,-1) -- (2,0);
		\draw[edge] (0,-2) -- (2,0);
  	
  		\node[vertex, label={left:$y$}] at (-2,0) {};
  		\node[vertex, label={left:$v_1$}] at (0,2) {};
  		\node[vertex, label={left:$v_2$}] at (0,1) {};
  		\node[vertex, label={left:$v_3$}] at (0,0) {};
  		\node[vertex, label={left:$v_4$}] at (0,-1) {};
  		\node[vertex, label={left:$v_5$}] at (0,-2) {};
  		\node[vertex, label={right:$x$}] at (2,0) {};
  	\end{tikzpicture}
  \caption{The graph $H_7$.} \label{HnFigure}
\end{figure}
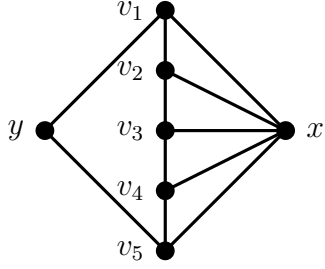  

The next theorem shows that the graphs in 
${\cal G}$ must have a very specific structure.  See Figure~\ref{NonInducedExamplesFigure} for examples of its members on $11$ vertices.
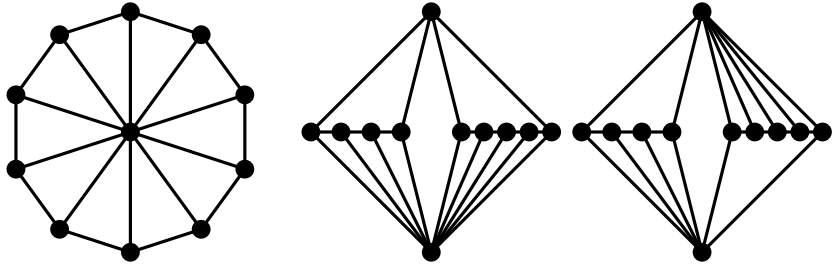
\begin{figure}
  \centering
  	\begin{tikzpicture}[scale=0.8]
  		\draw[edge] (0,2) -- (2*0.588,2*0.809) -- (2*0.951,2*0.309) -- (2*0.951,2*-0.309) -- (2*0.588,2*-0.809) -- (0,-2) -- (2*-0.588,2*-0.809) -- (2*-0.951,2*-0.309) -- (2*-0.951,2*0.309) -- (2*-0.588,2*0.809) -- cycle;
  		\draw[edge] (0,0) -- (0,2);
  		\draw[edge] (0,0) -- (2*0.588,2*0.809);
  		\draw[edge] (0,0) -- (2*0.951,2*0.309);
  		\draw[edge] (0,0) -- (2*0.951,2*-0.309);
  		\draw[edge] (0,0) -- (2*0.588,2*-0.809);
  		\draw[edge] (0,0) -- (0,-2);
  		\draw[edge] (0,0) -- (2*-0.588,2*-0.809);
  		\draw[edge] (0,0) -- (2*-0.951,2*-0.309);
  		\draw[edge] (0,0) -- (2*-0.951,2*0.309);
  		\draw[edge] (0,0) -- (2*-0.588,2*0.809);
  	
  		\node[vertex] at (0,0) {};
  		\node[vertex] at (0,2) {};
  		\node[vertex] at (2*0.588,2*0.809) {};
  		\node[vertex] at (2*0.951,2*0.309) {};
  		\node[vertex] at (2*0.951,2*-0.309) {};
  		\node[vertex] at (2*0.588,2*-0.809) {};
  		\node[vertex] at (0,-2) {};
  		\node[vertex] at (2*-0.588,2*-0.809) {};
  		\node[vertex] at (2*-0.951,2*-0.309) {};
  		\node[vertex] at (2*-0.951,2*0.309) {};
  		\node[vertex] at (2*-0.588,2*0.809) {};
  		
  		\draw[edge] (5+0,2) -- (5+-2,0);
  		\draw[edge] (5+0,2) -- (5+-0.5,0);
  		\draw[edge] (5+0,2) -- (5+2,0);
  		\draw[edge] (5+0,2) -- (5+0.5,0);
  		\draw[edge] (5+-2,0) -- (5+-0.5,0);
  		\draw[edge] (5+-2,0) -- (5+0,-2);
  		\draw[edge] (5+-1.5,0) -- (5+0,-2);
  		\draw[edge] (5+-1,0) -- (5+0,-2);
  		\draw[edge] (5+-0.5,0) -- (5+0,-2);
  		\draw[edge] (5+2,0) -- (5+0.5,0);
  		\draw[edge] (5+2,0) -- (5+0,-2);
  		\draw[edge] (5+1.625,0) -- (5+0,-2);
  		\draw[edge] (5+1.25,0) -- (5+0,-2);
  		\draw[edge] (5+0.875,0) -- (5+0,-2);
  		\draw[edge] (5+0.5,0) -- (5+0,-2);
  		
  		\node[vertex] at (5+0,2) {};
  		\node[vertex] at (5+-2,0) {};
  		\node[vertex] at (5+-1.5,0) {};
  		\node[vertex] at (5+-1.0,0) {};
  		\node[vertex] at (5+-0.5,0) {};
  		\node[vertex] at (5+2,0) {};
  		\node[vertex] at (5+1.625,0) {};
  		\node[vertex] at (5+1.25,0) {};
  		\node[vertex] at (5+0.875,0) {};
  		\node[vertex] at (5+0.5,0) {};
  		\node[vertex] at (5+0,-2) {};
  		
  		\draw[edge] (9.5+0,2) -- (9.5+-2,0);
  		\draw[edge] (9.5+0,2) -- (9.5+-0.5,0);
  		\draw[edge] (9.5+0,2) -- (9.5+2,0);
  		\draw[edge] (9.5+0,2) -- (9.5+1.625,0);
  		\draw[edge] (9.5+0,2) -- (9.5+1.25,0);
  		\draw[edge] (9.5+0,2) -- (9.5+0.875,0);
  		\draw[edge] (9.5+0,2) -- (9.5+0.5,0);
  		\draw[edge] (9.5+-2,0) -- (9.5+-0.5,0);
  		\draw[edge] (9.5+-2,0) -- (9.5+0,-2);
  		\draw[edge] (9.5+-1.5,0) -- (9.5+0,-2);
  		\draw[edge] (9.5+-1,0) -- (9.5+0,-2);
  		\draw[edge] (9.5+-0.5,0) -- (9.5+0,-2);
  		\draw[edge] (9.5+2,0) -- (9.5+0.5,0);
  		\draw[edge] (9.5+2,0) -- (9.5+0,-2);
  		\draw[edge] (9.5+0.5,0) -- (9.5+0,-2);
  		
  		\node[vertex] at (9.5+0,2) {};
  		\node[vertex] at (9.5+-2,0) {};
  		\node[vertex] at (9.5+-1.5,0) {};
  		\node[vertex] at (9.5+-1.0,0) {};
  		\node[vertex] at (9.5+-0.5,0) {};
  		\node[vertex] at (9.5+2,0) {};
  		\node[vertex] at (9.5+1.625,0) {};
  		\node[vertex] at (9.5+1.25,0) {};
  		\node[vertex] at (9.5+0.875,0) {};
  		\node[vertex] at (9.5+0.5,0) {};
  		\node[vertex] at (9.5+0,-2) {};
  	\end{tikzpicture}
  \caption{Graphs on $11$ vertices with $20$ edges and no proper  (not necessarily induced) subgraphs with minimum degree $3$.} \label{NonInducedExamplesFigure}
\end{figure}  
\begin{theorem}\label{NonInducedStructure} The family ${\cal G}$ consists
of  all wheels and those graphs that are formed, for some $i$ and $j$,  
from a copy of $H_i$ with connectors $x$ and $y$ and a copy of $H_j$ with connectors 
$x'$ and $y'$ by letting $x=x'$ and $y=y'$ or by letting $x=y'$ and $y=x'$.
\end{theorem}

For the proof we first recall some basic properties of 
graphs with no induced subgraphs of minimum degree $3$.

Recall from the introduction that the following lemma is easy to prove by induction.
\begin{lemma}\label{Deg3Subgraph}
Every graph on $n\geq2$ vertices with at least $2n-2$ edges contains an induced subgraph with minimum degree $3$.
\end{lemma}

For degree $3$-critical graphs, the induced subgraph of minimum degree $3$ (guaranteed by the previous lemma) must be the whole $G$. For these graphs, Erd\H os et al.
\cite{EFGS} presented a special ordering to the vertices.  Given an ordering $x_1, \dots, x_n$ of $V(G)$ we let the \emph{forward neighbourhood} of $x_i$, denoted  $N^+(x_i)$, be $N^+(x_i)= N(x_i)\cap \{x_{i+1}, \dots, x_{n}\}$.   The \emph{forward degree} of $x_i$  is $d^+(x_i)=|N^+(x_i)|$. 
The following lemma is essentially from \cite{EFGS}.  We prove it here in a
slightly stronger formulation. Notice that the lemma 
considers not just graphs from $\mathcal G$, but degree $3$-critical graphs in general.
We will make use of this in the next section.
\begin{lemma}\label{Ordering}
For every degree $3$-critical graph $G$ on $n$ vertices there is 
an ordering  $x_1, \dots, x_n$ of the vertices, such that the following hold.
\begin{enumerate}[(i)]
\item $d^+(x_1)=3$.
\item For $2\leq i\leq n-2$, $d^+(x_i)=2$.
\item $d^+(x_{n-1})=1$.
\item If furthermore $n \geq 7$, then $d(x_n) \geq 4$.
\end{enumerate}
\end{lemma}
\begin{proof}
We define $x_i$ recursively.  Let $x_1$ be a vertex of minimum degree in $G$. 
Suppose that we have already defined $x_{1}, x_{2}, \dots, x_{i}$.
Then we let $x_{i+1}$ be a vertex of minimal degree in $G-\{ x_{1}, \dots-, x_i\}$.

For (i), notice that the average degree of $G$ is less than $4$, so $d(x_1)\leq 3$. To see that $d(x_1)\geq 3$, notice that otherwise the graph $G-x_1$ would have at least 
$e(G)-2= 2(n-1) -2$ edges and Lemma~\ref{Deg3Subgraph} would imply the existence of
an induced subgraph of $G-x_1$ of minimum degree $3$, a contradiction to $G$ being degree
$3$-critical. Hence $d(x_1)=3$.

For (ii), we proceed by induction to show that for all $i$, $1\leq i\leq n-2$, we have 
$e(G-\{x_1, \dots, x_i\})=2(n-i)-3$. The case $i=1$ follows from (i).  
Let $i>1$ and assume $e(G-\{x_1, \dots, x_{i-1}\})=2(n-(i-1))-3$. 
First notice that degree $3$-criticality of $G$ implies 
both $d^+(x_{i})\leq 2$ and $e(G-\{x_1, \dots ,x_{i}\}) \leq 2(n-i)-3$.
Indeed, otherwise the minimum degree of the
induced subgraph $G-\{ x_1, \dots, x_{i-1}\}$ would be exactly $3$ or
$G-\{x_1, \dots ,x_{i}\}$ would contain an induced subgraph of minimum degree $3$
by Lemma~\ref{Deg3Subgraph}.
 On the other hand, $e(G-\{x_1, \dots ,x_{i}\}) = e(G-\{ x_1, \dots , x_{i-1}\})
-d^+(x_{i}) \geq 2(n-(i-1)) -3 - 2$ by induction, implying both $e(G-\{x_1, \dots , x_{i}\}) = 2(n-i)-3$ 
and $d^+(x_i) =2$.  

Part (iii) now follows from $e(G-\{x_1, \dots ,x_{n-2}\})=1$.

For (iv), assume that $n \geq 7$. Let $x_1,\dots, x_n$ be the ordering of the vertices of $G$ produced by the above procedure.
Notice that the graph $G[\{x_{n - 5}, x_{n - 4}, \dots, x_n\}]$ must contain a vertex $v$ of degree at least $4$ in $G[\{x_{n - 5}, x_{n - 4}, \dots, x_n\}]$ (since it has $6$ 
vertices and $9$ edges and contains a vertex of degree $2$ (here we use that 
$x_{n-5}\neq x_1$).  Since $d(v)\geq 4$, $v$ must be one of $x_{n - 3}$, $x_{n - 2}$, $x_{n - 1}$, or $x_n$.  The graph $G[\{x_{n - 3},  \dots, x_n\}]$ has $4$ vertices and $5$ edges, and so contains a vertex $x'_{n - 3} \neq v$ of degree $2$ in $G[\{x_{n - 3},  \dots, x_n\}]$. 
Let $x'_{n - 2}$, $x'_{n - 1}$ be the two vertices in $\{x_{n - 3}, x_{n - 2}, x_{n - 1}, x_n\}\setminus\{v, x'_{n - 3}\}$ in an arbitrary order.  Since 
$G[\{ x'_{n-2}, x'_{n-1}, v\}]$ spans a triangle, the ordering of $G$ given by 
$x_1, x_2, \dots x_{n - 5}, x_{n - 4}, x'_{n - 3}, x'_{n - 2}, x'_{n - 1}, v$ 
satisfies (i) -- (iv).
\end{proof}

\begin{proof}[Proof of Theorem~\ref{NonInducedStructure}]
First we show that if $G$ is a wheel or a graph formed from gluing $H_i$ and $H_j$ together, then $G$ is in $\mathcal G$.
If $G$ has a subgraph  $H$ of minimum degree $3$ and vertex 
$v\in V(H)$ with $d_G(v)=3$, then the three neighbours of  $v$ must all be in $H$. Hence the connected components of the induced subgraph of $G$ on its vertices of degree $3$ 
must either be fully contained in $H$ or fully missing.
Wheels have only one such component, and graphs formed from gluing $H_i$ and $H_j$ together as in the theorem have two such components.
Using this, it is easy to check that these graphs have no proper subgraphs of 
minimum degree $3$. 

For the reverse direction
let $G$ be an $n$-vertex graph with $2n-2$ edges  and no proper (not necessarily induced) subgraphs with minimum degree $3$.  
From Lemma~\ref{Ordering}, we have that $\delta(G)\geq 3$.
We formulate the property of $G$ that will be most important for us.

\begin{observation}\label{ObservationDegree4}
 The graph $G$ does not have 
two adjacent vertices of degree $\geq 4$. 
\end{observation}
Indeed, the removal of the edge between two vertices of degree $4$ 
would create a proper subgraph of $G$ minimum degree $3$,
a contradiction.

If $|G|\leq 6$, then it is easy to check (say by considering the ordering given in Lemma~\ref{Ordering}) 
that $G$ must be a wheel or the graph obtained by the gluing of two copies of $H_4$. 
Therefore, let us assume that we have $|G|\geq 7$.  

First we show that if $G$ is not a wheel, then it contains a copy of $H_m$ 
for some $m\geq 4$ with a certain structure to its internal vertices.  
\begin{claim}\label{HmSubgraph}
Either $G$ is a wheel or $G$ has an induced subgraph $H_m\subseteq G$ for some $m\geq 4$, such that  none of the internal vertices of $H_m$ have 
neighbours in $G - V(H_m)$. 
\end{claim}

\begin{proof}
Consider the ordering $x_1,\dots,x_n$ of the vertices of $G$ as given by Lemma~\ref{Ordering}.   
Let $k$ be the smallest integer such that 
$x_n$ is adjacent to every vertex in $\{ x_{k+1}, \ldots, x_{n-1}\}$.
Note that $k\in \{0, 1, \ldots n-3\}$, since by part (ii) and (iii) of 
Lemma~\ref{Ordering}, $x_{n-2}$ and $x_{n-1}$ are adjacent to $x_n$.
We will show that if $k=0$ then $G$ is a wheel and otherwise the subgraph 
$G[\{x_{k}, \dots, x_n\}]$ is the sort of
copy of $H_{n-k+1}$ that we need, with connectors $x=x_n$ and $y=x_{k}$. 
 
We plan to reconstruct 
$G[\{ x_{\ell}, \ldots, x_n\}]$ from the trivial graph on $\{ x_n\}$
by adding back one-by-one the vertices $x_i$ for each 
$i=n-1, n-2 \ldots, \ell$ (in reverse order), 
together with their incident edges to $\{ x_{i+1}, \ldots , x_n\}$. 

First we show by backward induction that 
the induced subgraph $G[\{ x_{i}, \ldots , x_{n-1}\}]$ is a path $R_i$ for every 
$i=\max\{ k+1, 2\}, \ldots , n-2$.
Indeed, for every $i=\max\{ k+1, 2\}, \ldots , n-2$ the vertex $x_i$ is adjacent to $x_n$ 
and by part (ii) of Lemma~\ref{Ordering} to exactly one other vertex $x_j$ in
$\{ x_{i+1}, \ldots , x_{n-1}\}$.  By part (iv) of Lemma~\ref{Ordering} the degree of 
$x_n$ in $G$ is at least $4$ and since $x_jx_n\in E(G)$, Observation~\ref{ObservationDegree4} implies
that the degree of $x_j$ in $G[\{ x_{i+1}, \ldots , x_{n-1}\}]$ must be at most one. 
So $x_j$ is one of the 
endpoints of $R_i$, thus giving rise to a path $R_{i-1}$ that is induced on 
$\{ x_i, \ldots , x_{n-1}\}$. 

Now we separate into two cases. 

If $k >0$, then we have  that $G[\{ x_{k+1}, \ldots, x_{n-1}\}]$ is a path $R_k$ 
with all its vertices adjacent to $x_n$. Since $x_k$ is not adjacent to $x_n$,
both of its forward neighbours must be in $\{ x_{k+1}, \ldots , x_{n-1}\}$.
If any of these neighbours would be a vertex $x_j$, $k < j < n$,  
with degree at least $2$ in $G[\{ x_{k+1}, \ldots, x_{n-1}\}]$, then we get a contradiction 
from Observation~\ref{ObservationDegree4} as $x_jx_n\in E(G)$.
Hence $x_k$ must be adjacent exactly to the two endpoints of the path $R_k$ and then
$G[\{x_{k}, \dots, x_n\}]$ is a 
copy of $H_{n-k+1}$ with connectors $x=x_n$ and $y=x_{k}$ as we promised.
Observe furthermore that 
there cannot be any additional edges between any $x_i\in \{x_{k+1}, \dots, x_{n-1}\}$ 
and $V(G)\setminus \{ x_k, \ldots , x_n\}$, 
since otherwise the degree of $x_i$ in $G$ would be at least $4$ providing a contradiction from Observation~\ref{ObservationDegree4}
as $x_ix_n\in E(G)$.

If $k=0$, then $G[\{ x_{2}, \ldots, x_{n-1}\}]$ is a path $R_2$ 
with all its vertices adjacent to $x_n$. Again, none of the neighbours $x_j$ of
$x_1$ can be an internal vertex of $R_2$, otherwise we obtained a contradiction from 
Observation~\ref{ObservationDegree4} since $x_1x_n$ and  $x_jx_n$ are both edges  of $G$. 
Recall that $x_1$ has three neighbours (part (i) of Lemma~\ref{Ordering}).
These then must be the two endpoints of $R_2$ and $x_n$, giving rise to a wheel with
center $x_n$.  
\end{proof}

Given a copy of $H_m$ contained in $G$, we define $G/H_m$ to be the graph formed out of $G$ by removing the internal vertices of 
$H_m$, and joining the connectors of $H_m$ by an edge.  It turns out that if $H_m$ has the structure produced by 
Claim~\ref{HmSubgraph}, then the graph $G/H_m\in {\cal G}$, so we will be able to apply induction.
\begin{claim}\label{HmContraction}
Suppose that  graph $G\in {\cal G}$ has an induced subgraph $H_m\subseteq G$ for some $m$, 
such that none of the internal vertices of $H_m$ have neighbours in $G\setminus V(H_m)$.  
Then $G/H_m\in {\cal G}$.
\end{claim}
\begin{proof}
Let $x$ and $y$ be the connectors of $H_m$.  By the assumptions of the lemma and the definition of $G/H_m$, the only edges which were present in $G$ and are not present in $G/H_m$ are the $2m-3$ edges of $H_m$.  The only new edge in $G/H_m$ is the edge $xy$.  From the definition of $G/H_m$, we have $|G/H_m|=|G|-m+2$.  Combining this with $e(G)=2|G|-2$, we obtain $e(G/H_m)=e(G)-2m+4=2|G|-2m+2=2|G/H_m|-2$.

We will show that for every proper subgraph $K\subsetneq G/H_m$, we have $\delta(K)\leq 2$.  
If $K$ does not contain the edge $xy$, then $K$ is also a proper subgraph of $G$, and then, since $G\in {\cal G}$, 
$K$ must satisfy $\delta(K)\leq 2$.  Suppose now that $K$ does contain the edge $xy$.  
Let $K'$ be the graph formed from $K$ by removing the edge $xy$, 
and adding the vertices and edges of $H_m$. Since $G\in {\cal G}$, the proper subgraph $K' \subsetneq G$ must contain a 
vertex $v$ of degree at most $2$.  
The vertex $v$ cannot be one of the internal vertices of $H_m$, since by the definition of $H_m$, all internal vertices have degree $3$.   
So $v$ is also a vertex of $K$. But the degree of any vertex of $V(K)$ in $K'$ is at least as large as its degree in $K$ 
(in fact, unless $u=x$ or $u=y$,  the degree of $u$ in $K$ is equal to its degree in $K'$). 
Hence the vertex $v\in V(K)$ has degree at most $2$ in $K$ as well.
\end{proof}

Now we are ready to complete the proof of the theorem using induction on $|G|$. 
The initial cases are when $|G|\leq 6$, and are easy to check by hand.
Let $G\in {\cal G}$ be a graph on $n\geq 7 $ vertices. 
We will show that $G$ possesses one of the two structures given in the theorem.

If $G$ is not a wheel, then by Claim~\ref{HmSubgraph} $G$ 
contains an induced copy of $H^*$ of $H_m$ such that  the internal vertices of 
$H^*$ have no neighbours outside of $H^*$.
By Claim~\ref{HmContraction}, $G/H^*\in {\cal G}$. 
Hence, by induction, $G/H^*$ is either a wheel or 
is a graph formed by gluing together a copy of $H_i$ 
with connectors $x$ and $y$ and a copy of $H_j$ 
with connectors $x'$ and $y'$, for some $i,j\geq 4$.

First consider the case when $G/H^*$ is a wheel with center $c$ and outside vertices $w_1, \dots, w_k$. 
Recall that there is an edge in $G/H^*$ between the two connectors of $H^*$.  

Suppose first that the connectors of $H^*$ are $c$ and $w_i$  for some $i$.  
In this case, $G$ is a graph formed from $H_{k+1}$ and $H_m$ by identifying the connectors of the two graphs.  
Indeed, this follows from the fact that removing the edge $cw_i$ from the wheel 
gives a copy of $H_{k+1}$ and from the fact that the internal vertices of $H^*$ have no neighbours outside of $H^*$.  

Suppose now that the connectors of $H_m$ are two adjacent outside vertices of the wheel,
say $w_1$ and $w_2$. 
If $k=3$ then the graph $G/H^*$ is just the complete graph on $4$ vertices, so,
as before, $G$ is a graph formed from $H_{4}$ and $H_m$ with connectors $w_1$ and $w_2$.  So suppose that $k\geq 4$.  This ensures that $d(c)\geq 4$ in $G$.  We also have $d(w_{2})\geq 4$ in $G$ since $w_{2}$ must be connected to $c$, $w_{3}$, as well as all the internal vertices of $H^*$ (of which there are at least $2$).  But this gives a contradiction by Observation~\ref{ObservationDegree4}, since $cw_2$ is an edge of $G$.

Now, consider the case when $G/H^*$ is a graph formed by gluing together an 
$H_i$ and an $H_j$ at their connectors. 
Recall that there is an edge in $G/H^*$ between the two connectors of $H^*$.
Suppose, without loss of generality, that this edge is in $H_i$.  
Let $x$ and $y$ be the connectors of $H_i$ and let 
$v_1, \dots, v_{i-2}$ be its internal vertices. Since $xy\not\in E(H_i)$, 
one of the connectors of $H^*$ must be an internal vertex of $H_i$.
If any internal vertex of $H_i$ which is a connector of $H^*$ is adjacent in $G$ to
any vertex of $\{ x,y\}$  which is not a connector of $H^*$, then
we immediately get a contradiction by Observation~\ref{ObservationDegree4} since both of these vertices
have degree at least $ 4$.
Otherwise, for the internal vertex $v_t$ of $H_i$ which is a connector of $H^*$ 
we must have $1 <t < i-2$, and the other connector vertex must be $x$.
Then the proper 
subgraph  $G - \{v_{1}, v_{2}, \dots,v_{t-1}\}$ has minimum degree $3$,  
contradicting our assumption of $G$ having no such subgraphs.  
This completes the proof of the inductive step and the theorem.
\end{proof}

It is an easy exercise to check that the graphs given in
Theorem~\ref{NonInducedStructure} are pancyclic and hence
Theorem~\ref{NonInduced} follows.

\section{Finding a $6$-cycle}\label{C6Section}
\begin{proposition}
Every degree $3$-critical graph $G$ with $n\geq 6$ contains a $C_6$.
\end{proposition}
\begin{proof}
By Lemma~\ref{Deg3Subgraph} we have $\delta(G)\geq 3$.

Let us use 
Lemma~\ref{Ordering} to obtain an ordering $x_1, \dots, x_n$ of the vertices of $G$. 
By part (ii) and (iii) and using $|G|\geq 5$, the graph induced by the last four vertices is a $K_4$ minus an edge.
Let us assume without loss of generality that the missing edge is $x_{n-3}x_{n-2}$, that
is, both $x_{n-1}$ and $x_n$ have degree $3$ in $G[\{ x_{n-3}, x_{n-2}, x_{n-1}, x_n\}]$.

Now let $t\leq n-4$ be the largest index for which the forward neighbourhood of the 
vertex $x_t$ is not $\{ x_{n-1}, x_n\}$ ($t$ exists because, for example ``$1$'' is such an index).

First let us suppose that  $x_t$ has two forward neighbours $x_i$ and $x_j$ outside of
$\{x_{n-1}, x_n\}$.  By the definition of $x_t$ we have that $x_i$ and $x_j$ are both 
adjacent to $x_{n-1}$ and $x_n$. 
Let $m \in [n] \setminus \{ n, n-1, i, j, t\}$ be the largest index
such that the forward neighbourhood of $x_m$ is {\em not} equal to $\{ x_i, x_j\}$ ($m$ exists since $|G|\geq 6$).
Note that if $\{ i, j, t\} \neq \{ n-2, n-3, n-4\}$, then we have $m \geq n-4$ and 
the forward neighbourhood of $x_m$ is $\{ x_n, x_{n-1}\}$. 
Thus $x_{n-1}x_mx_nx_jx_tx_i$ is a six-cycle (see Figure~\ref{6CycleFig1}).
If $\{ i, j, t\} = \{ n-2, n-3, n-4\}$, then the graph 
$G[\{ x_n, \ldots , x_{m+1}\}]$ (see Figure~\ref{6CycleFig1}) has the property
that any pair of vertices, but $\{ x_{n-2}, x_{n-3}\}$ have a path of 
length four between them. Thus the addition of $x_m$ will create a six-cycle.

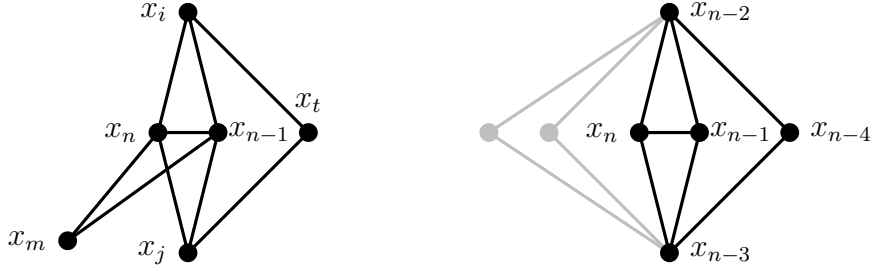
\begin{figure}
  \centering
  	\begin{tikzpicture}[scale=0.8]
		\draw[edge] (0,2) -- (-0.5,0);
		\draw[edge] (0,2) -- (0.5,0);
		\draw[edge] (0,2) -- (2,0);
		\draw[edge] (-0.5,0) -- (0.5,0);
		\draw[edge] (-0.5,0) -- (0,-2);
		\draw[edge] (0.5,0) -- (0,-2);
		\draw[edge] (0,-2) -- (2,0);
		\draw[edge] (-0.5,0) -- (-2,-1.8);
		\draw[edge] (0.5,0) -- (-2,-1.8);

  		\node[vertex, label={left:$x_i$}] at (0,2) {};
  		\node[vertex, label={left:$x_n$}] at (-0.5,0) {};
  		\node[vertex, label={[inner sep=0]right:$x_{n - 1}$}] at (0.5,0) {};
  		\node[vertex, label={left:$x_j$}] at (0,-2) {};
  		\node[vertex, label={above:$x_t$}] at (2,0) {};
  		\node[vertex, label={left:$x_m$}] at (-2,-1.8) {};

		\draw[edge, gray!50] (8+0,2) -- (8-2,0) -- (8+0,-2);
		\draw[edge, gray!50] (8+0,2) -- (8-3,0) -- (8+0,-2);
		\draw[edge] (8+0,2) -- (8+-0.5,0);
		\draw[edge] (8+0,2) -- (8+0.5,0);
		\draw[edge] (8+0,2) -- (8+2,0);
		\draw[edge] (8+-0.5,0) -- (8+0.5,0);
		\draw[edge] (8+-0.5,0) -- (8+0,-2);
		\draw[edge] (8+0.5,0) -- (8+0,-2);
		\draw[edge] (8+0,-2) -- (8+2,0);

  		\node[vertex, fill=gray!50] at (8-2,0) {};
  		\node[vertex, fill=gray!50] at (8-3,0) {};
  		\node[vertex, label={right:$x_{n-2}$}] at (8+0,2) {};
  		\node[vertex, label={left:$x_n$}] at (8+-0.5,0) {};
  		\node[vertex, label={[inner sep=0]right:$x_{n - 1}$}] at (8+0.5,0) {};
  		\node[vertex, label={right:$x_{n-3}$}] at (8+0,-2) {};
  		\node[vertex, label={right:$x_{n-4}$}] at (8+2,0) {};
    	
  	\end{tikzpicture}
  \caption{The two possible configurations which can occur in the case when $x_t$ has two forward neighbours $x_i$ and $x_j$, outside of $\{x_{n-1}, x_n\}$. The grey vertices represent ones which may or may not be present.} \label{6CycleFig1}
\end{figure}  

Suppose now that $x_t$ has exactly one forward neighbour $x_i$, with 
$t+1\leq i \leq n-2$, outside of $\{x_{n-1}, x_n\}$.  
Without loss of generality let $x_n$ be a neighbour of $x_t$ in $\{ x_n, x_{n-1}\}$.
By the definition of $x_t$ we have that $x_i$ is adjacent to both $x_{n-1}$ and $x_n$.  
If $i=n-2$, let us define $s:=n-3$, and otherwise let $s:=n-2$.
Let $m$ be the smallest index such that 
the forward neighbourhood of $x_m$ is neither $\{x_i, x_n\}$ nor $\{x_{n-1}, x_n\}$
($m$ exists since  the index ``$1$'' is certainly of that kind). Then the structure of the 
graph $G[\{ x_{m+1}, \ldots , x_n\}]$ looks like the one 
in Figure~\ref{6CycleFig2}. Observe that for any pair of vertices in such a graph, 
but the pairs $\{ x_{n-1}, x_n\}$ and $\{ x_n, x_i\}$, there is a path of length 
four between them. Hence no matter where 
the two forward neighbours $x_j$ and $x_l$ of $x_m$, with $\{j,l\} \neq \{ n-1, n\}, \{ n,i\}$, are, they close a six-cycle.

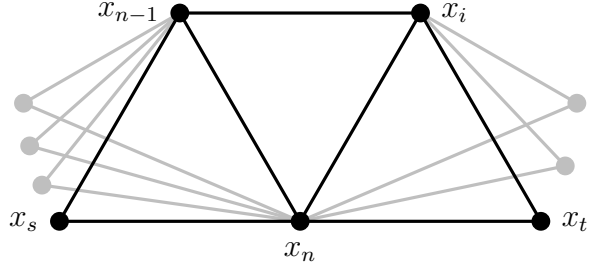
\begin{figure}
  \centering
  	\begin{tikzpicture}[scale=0.8]
		\draw[edge, gray!50] (-2,3.464) -- (-4.291,0.602) -- (0,0);
		\draw[edge, gray!50] (-2,3.464) -- (-4.495,1.254) -- (0,0);
		\draw[edge, gray!50] (-2,3.464) -- (-4.598,1.964) -- (0,0);
		\draw[edge, gray!50] (2,3.464) -- (4.405,0.921) -- (0,0);
		\draw[edge, gray!50] (2,3.464) -- (4.598,1.964) -- (0,0);
		\draw[edge] (-2,3.464) -- (2,3.464);
		\draw[edge] (-2,3.464) -- (-4,0);
		\draw[edge] (-2,3.464) -- (0,0);
		\draw[edge] (2,3.464) -- (0,0);
		\draw[edge] (2,3.464) -- (4,0);
		\draw[edge] (-4,0) -- (4,0);
		
		\node[vertex, fill=gray!50] at (-4.291,0.602) {};
		\node[vertex, fill=gray!50] at (-4.495,1.254) {};
		\node[vertex, fill=gray!50] at (-4.598,1.964) {};
		\node[vertex, fill=gray!50] at (4.405,0.921) {};
		\node[vertex, fill=gray!50] at (4.598,1.964) {};
  		\node[vertex, label={left:$x_{n - 1}$}] at (-2,3.464) {};
  		\node[vertex, label={right:$x_i$}] at (2,3.464) {};
  		\node[vertex, label={left:$x_s$}] at (-4,0) {};
  		\node[vertex, label={below:$x_n$}] at (0,0) {};
  		\node[vertex, label={right:$x_t$}] at (4,0) {};
  	\end{tikzpicture}
  \caption{The possible induced subgraphs $G[x_{m+1}, \ldots , x_n\}]$ in the case when $x_t$ has exactly one forward neighbour $x_i$ outside of $\{x_{n - 1}, x_n\}$.  
The unlabeled vertices may or may not be there.} \label{6CycleFig2}
\end{figure}  

\end{proof}

\section{Concluding remarks}\label{RemarksSection}
In Theorem~\ref{Counterexamples} we constructed degree $3$-critical graphs with no $23$-cycles.  One could ask whether longer cycles could be forbidden as well.  It is easy to use our method to construct sequences of degree $3$-critical graphs with no $m$-cycles for any odd $m\geq 23$. Indeed, combining Proposition~\ref{SteppingUp} with Theorem~\ref{thm:20-avoiding} shows that there are $2k$-avoiding sequences for all $k\geq 10$. Then Lemmas~\ref{CycleLengths} and~\ref{PathLengths} give us degree 3-critical graphs with no cycles of length $2k+3$ for all $k\geq 10$.
It would be interesting to determine the shortest cycle length $\ell$ for which there exist an infinite sequence of degree $3$-critical graphs with no cycle of length $\ell$. From the results in this paper we see that $\ell$ must be between $7$ and $23$.

In this paper we were only able to find infinite sequences of degree $3$-critical graphs which do not contain \emph{odd} cycles. It is not clear whether even cycles can be forbidden in the same way.  We pose the following problem.
\begin{problem}
Is there a function $C(n)$ tending to infinity such that every degree $3$-critical graph on $n$ vertices contains cycles of all lengths $4, 6, 8, \dots, 2C(n)$.
\end{problem}

Another natural extremal question concerns the {\em number} of different cycle length. 
A construction due to Bollob\'as and Brightwell \cite{BB} gives degree $3$-critical graphs
 with no cycles of length greater than $4\log_2 n+ O(1)$.  Their construction is just 
 the graph $G(T_d)$ where $T_d$ is the $1$-$3$-tree having a root with each of his three
 subtrees being a perfect binary tree of depth~$d$.  We conjecture 
 that these graphs give the smallest number of cycle lengths amongst all degree $3$-critical graphs on $n$ vertices.
\begin{conjecture}\label{CycleNumberConjecture}
Every degree $3$-critical graph on $n$ vertices contains cycles of  at least 
$3\log_2 n + O(1)$ distinct lengths.
\end{conjecture}

A similar conjecture could be made about leaf-leaf paths in trees.
\begin{conjecture}\label{PathNumberConjecture}
Every $1$-$3$ tree has leaf-leaf paths of at least $\log_2 n$ distinct lengths.
\end{conjecture}

In this paper we have shown that for $d\geq 20$, it is impossible to guarantee that a sufficiently large $1$-$3$ tree $T$ contains a leaf-leaf path of length $d$.  However, perhaps it is the case that in a sufficiently large $1$-$3$ tree, there are leaf-leaf paths of ``many'' short lengths.
\begin{conjecture}
There is a constant $\alpha>0$ and a function $C(n)$ tending to infinity 
such that every $1$-$3$ tree of order $n$ contains  at least $\alpha C(n)$ of distinct leaf-leaf path lengths between $0$ and $C(n)$.
\end{conjecture}

\bibliography{Cycles}
\bibliographystyle{abbrv}
\end{document}